\documentclass[a4paper]{article}
\usepackage{amscd,amssymb,amsmath,amsthm}
\usepackage[utf8]{inputenc}
\usepackage[english]{babel}
\usepackage{graphicx}
\usepackage{booktabs}
\usepackage{color}
\usepackage[lofdepth,lotdepth]{subfig}
\usepackage{hyperref}
\usepackage{todonotes}
\newcommand{\Z}{\mathbb{Z}}
\newtheorem{thm}{Theorem}[section]
\newtheorem{defn}{Definition}[section]
\newtheorem{lemma}{Lemma}[section]

\newtheorem{cor}{Corollary}[section]
\usepackage{ulem}

\theoremstyle{definition}

\def\N{{\mathbb N}}

\def\Z{{\mathbb Z}}
\def\R{{\mathbb R}}

\newcommand{\F}{{\mathcal F}}

\newcommand{\e}{\mathrm e}

\newcommand{\Cov}{\mbox{$\mathrm Cov$}}

\newcommand{\red}{\textcolor{red}}

\begin{document}
\title{Continuity of the extremal decomposition of the free state for finite-spin models on Cayley trees} 
\author{Loren Coquille\footnote{Universit\'e Grenoble Alpes, CNRS, Institut Fourier, F-38000 Grenoble, France; loren.coquille@univ-grenoble-alpes.fr}, Christof Külske\footnote{Ruhr-Universit\"at Bochum, Fakult\"at f\"ur Mathematik, D44801 Bochum, Germany; Christof.Kuelske@ruhr-uni-bochum.de}, 
Arnaud Le Ny\footnote{Laboratoire d'Analyse et de Mathématiques Appliquées,  UMR CNRS 8050, UPEC, 61 Avenue du G\'en\'eral de Gaulle,  94010 Cr\'eteil cedex, France; arnaud.le-ny@u-pec.fr}
}		
		\maketitle

	\begin{abstract} We prove the continuity of the extremal decomposition measure of 
		the free state of low temperature Potts models, and more generally of ferromagnetic finite-spin models, on a regular tree, {including general clock models}.
		The decomposition is supported on uncountably many 
		inhomogeneous extremal states, that we call {\it glassy states}.
		The method of proof provides explicit concentration 
		bounds on {\it{branch overlaps}}, which 
		play the role of an order parameter  
		for typical extremals.
		The result extends to the counterpart of the free state  (called {\it central state}) in a wide range of models which have no symmetry,  
		{allowing also the presence of}  
		sufficiently small field terms. 
		Our work shows in particular 
		that {the decomposition of central state{s} {i}nto uncountably many
			glassy states in finite-spin models on trees at low temperature is {a} 
			generic {phenomenon}, 
			and does not rely on symmetries of the Hamiltonian. }
	\end{abstract}
\vspace{0.3cm}


\newpage

\tableofcontents

\newpage

\section{Introduction}

Models of statistical mechanics supported on trees and more generally non-amenable graphs, 
are known to show rather interesting behavior with multiple critical values, which displays 
complex phenomenology and challenges for a rigorous understanding. 

The free state of the Ising model in zero external magnetic field on a regular tree, obtained as the infinite-volume Gibbs measure with free (open) boundary conditions, is such an example. 
This homogenous (tree-automorphism invariant) Gibbs measure indeed displays a complex structure at low temperature, which took a long history to understand.

First recall \cite{Bod06} that the free state $\mu^{\Z^d}_{\beta}$ of the Ising model on the lattice $\Z^d$ at inverse temperature $\beta$   
decomposes into the symmetric convex combination 
\begin{equation}\label{free-lattice}\begin{split}
\mu^{\Z^d}_{\beta}=(\mu^{+,\Z^d}_{\beta}+\mu^{-,\Z^d}_{\beta})/2, \quad\text{for all }\beta>0.
\end{split}
\end{equation}
where $\mu^{+,\Z^d}_{\beta}$ (respectively $\mu^{-,\Z^d}_{\beta}$) is obtained as the infinite-volume Gibbs measure with with $+$ (respectively $-$)-boundary conditions.
The extremal decomposition of this form becomes non-trivial 
{below} the critical temperature of the model, i.e. 
for all $\beta>\beta_c^{\Z^d}:=\inf \{\beta>0, \mu^{+,\Z^d}_{\beta}\neq\mu^{-,\Z^d}_{\beta}\}$. By abstract Gibbs theory \cite{HOG} for any spin-model on countable graphs, a unique decomposition into extremal (pure, non\red{-}decomposable) 
infinite-volume Gibbs measures always holds. 

On the tree $\mathcal T^d$ with $d+1>2$ nearest neighbors however, 
\begin{equation}\label{free-tree}\begin{split}
\mu^{\mathcal T^d}_{\beta}\neq 
(\mu^{+,\mathcal T^d}_{\beta}+\mu^{-,\mathcal T^d}_{\beta})/2
\end{split}
\end{equation}
in the whole non-uniqueness region $d(\tanh\beta)>1$ corresponding to 
$\beta>\beta_c^{\mathcal T^d}$ \cite{Pres74}.
Moreover, the free state is  non-extremal in the even lower temperature region $d(\tanh\beta)^2>1$, which 
was shown in the sequence of works {\cite{hig77,BL90, BRZ95,ioffe2, iof96}.  

Note that Ising models on a branching plane $\mathcal T^d\times\Z$ were investigated in \cite{NW90} and display three parameter regimes corresponding to: (i) uniqueness, (ii) non-uniqueness with {tree-like} structure \eqref{free-tree} and (iii) non-uniqueness with {lattice-like} structure \eqref{free-lattice}.

Understanding the extremal decomposition of the Ising free state on a tree is known 
to be particularly complex as it involves uncountably many Gibbs measures on which the extremal 
decomposition measure is supported, 
see the much more recent work of \cite{GMRS20}. These states have a broken translational symmetry, 
and show characteristics of glassy behavior. 

To the best of our knowledge there are no result concerning the extremal decomposition 
measure of the Ising free state in an external field, nor of the one of the Potts 
free state without a field or in a field, and how sensitive the decomposition measure depends on particularities 
and symmetries of the model. 
Our purpose in the present work is to investigate this complexity in the general framework of 
ferromagnetic finite-spin models with nearest neighbor interactions.


In the  paper \cite{CoKuLe22}, we prove the existence of an uncountable family of extremal inhomogeneous states for a large class of models on regular trees, including finite-spin models. More precisely, we give an explicit description of a family of local ground state configurations $ \Omega_{GS}$, which have a sparse enough set of broken bonds, and which give rise to extremal low temperature states, in the sense that for any $\omega\in\Omega_{GS}$, the weak limit of finite volume measures with boundary condition $\omega$ exists and is an extremal state.  Note that many other extremal states exist, see discussions in \cite{CoKuLe22, KRK14,GRS15} and references therein.

The purpose of the current paper is to show that uncountably many states, which can be seen as perturbations of the above inhomogeneous states, enter the extremal decomposition of the free state. The result extends to the counterpart of the free state  (called {\it central state}), in a wide range of asymmetric models {\it e.g.}\ where small {in}homogenous field terms are added. 
Our approach {has been} inspired by the work of Gandolfo, Maes, Ruiz and Shlosman \cite{GMRS20}, {who} discuss these questions in the case of the Ising model in zero external field. {Our approach 
is {nevertheless} essentially different as we develop a new method based on \textit{concentration 
of branch overlaps}, see below. }\\

In the Ising case, soon after the characterization of the homogeneous Gibbs measures as Markov chains by Spitzer \cite{Spi75}, Higuchi \cite{hig77} proved the non-extrema\-lity of the 'third Markov chain' (the free {state}) for $d(\tanh\beta)^2>1$. These works were complemented by Bleher {\it et al.}\ \cite{BRZ95} and  Ioffe \cite{ioffe2} who proved extremality of the free state for $d(\tanh\beta)^2<1$.  In the latter, and already in \cite{CCST, EG83}, the fluctuations leading to non extremality have been related to fluctuations of the spin-glass order paramater on trees used as a discriminating tail observable, the so-called {\it Edwards-Anderson parameter}, which is nothing but the variance of the root magnetization.} 
{Note that non-extremality of the free state was also proved by Markov chains Kesten-Stigum techniques in \cite{MS79}, who also conjectured independently the extremality at intermediate temperatures.}

For general finite spin models, {including $q$-clock models}, we introduce a generalised version of the Edwards-Anderson parameter measuring the non-degeneracy of the law of the root-spin when the boundary condition at infinity is sampled according to the free/central state $\mu$: 
$$q_{EA}=\frac1q\sum_{a\in \{1,\ldots,q\}}{\rm Var}_\mu\bigl({\pi( {\sigma_0=a} | \cdot})\bigr)$$
were the {$\pi$}-kernel $\pi(\cdot | \omega)$ is the Gibbs-measure with boundary-condition $\omega$ 
at infinity, which exists and is extremal for $\mu$-almost every $\omega$ (for precise definitions, see Section \ref{sec-DLR} below). 
We prove $q_{EA}$ to be strictly positive at large enough $\beta$ (see Theorem \ref{thm-EA}), implying that the free/central state is not extremal at low temperature.
From an information-theoretic view, this means the model is reconstruction-solvable \cite{Mossel2001} 
as the $\pi$-kernel is able to restore information which was sent from the root to infinity 
by means of the measure $\mu$.
\\

Our main results concern the extremal decomposition of the free/central state $\mu$.
The DLR equations imply
\begin{align*}
	\mu(\cdot)=\int_\Omega \pi(\cdot|\omega) d\mu(\omega)
	=\int_{{{\rm ex} \mathcal G(\gamma)}}\nu(\cdot){\rm d}\alpha_\mu(\nu)
\end{align*}
where $\alpha_\mu$ is the {\textit{extremal decomposition measure}} of the free state. 
In Theorem \ref{thm-sing-ex}, we prove that if we sample independently at random two boundary conditions $\omega$ and $\omega'$ under the free state $\mu$, then the states $\pi(\cdot |\omega)$ and $\pi(\cdot |\omega')$ are almost surely singular with respect to each other:
\begin{align*}
	\mu \otimes \mu
	(\{(\omega,\omega') : \pi(\cdot |\omega) \perp \pi(\cdot | \omega')\})=1.
\end{align*} 
which implies that the measure $\alpha_\mu$ has no atom. The extremal decomposition of the free state $\mu$ is thus supported on an uncountable family of extremal states.\\
To this purpose, we introduce the following tail measurable observable
\begin{equation*}\begin{split}
	&\underline\phi^\omega=\liminf_{n\uparrow\infty}\frac{1}{|\Lambda_n|}\sum_{v\in \Lambda_n}1_{\sigma_v=\omega_v}\cr
	\end{split}
	\end{equation*}
which we call {\it branch overlap}, and which measures how much a configuration $\sigma$ agrees with $\omega$ on a (sparse enough) sequence of vertices lying on a branch of the tree.
We prove that $\underline\phi^\omega$  has different expectations under $\pi(\cdot |\omega)$ and $\pi(\cdot |\omega')$, which is enough to imply almost sure singularity. 
More precisely, we prove that $\underline\phi^\omega$ has an expectation tending to 1 under $\pi(\cdot |\omega)$, and tending to $\sum_{a}\mu(\sigma_0=a)^2$ (which equals $1/q$ in the case of the $q$ state Potts model) under $\pi(\cdot |\omega')$, as $\beta\to\infty$, with explicit bounds. This can be viewed as a quantitative statement of some "boundary condition resampling chaos".\\

The paper is organized as follows: 
In Section \ref{sec-tools} {we} first recall the necessary backgrounds on Gibbs measures, extremal decomposition{s} and Markov chains on Cayley trees. We then state our results in Section \ref{sec-results} and provide the proofs in Section \ref{sec-proofs}.

\section{Definitions and tools}\label{sec-tools}

\subsection{Models}

Let $\mathcal{T}^d=(V,E)$ denote the Cayley tree of order $d$, on which any vertex $v\in V$ has exactly $d+1$ neighbors. {Pairs $\{v,w\}\in E$ of nearest-neighbors are  written $v \sim w$.} 
To any vertex $v\in V$, we attach a spin, which is a random variable $\sigma_v$ taking values in $\Omega_0={\Z_q=\{0,\dots, q-1\}}$, for $q \in \{2,3,\ldots\}$. Let $\Omega_0$ be equipped with the $\sigma$-algebra $\mathcal{E}=\mathcal P(\Omega_0)$ of all subsets of $\Omega_0$.
We are interested in probability measures on the product space $(\Omega,\mathcal F)=(\Omega_0^V, \mathcal{E}^{\otimes V})$. 
For any subset  {$\Lambda\subset V$ and $\omega\in\Omega$ we define the finite-volume configuration (or projection}) ${\sigma_{\Lambda}(\omega)}=\omega_\Lambda=(\omega_v)_{v\in\Lambda}$. 
If $\Lambda\subset V$ is a finite subset, we {also} write $\Lambda\Subset V$. We write $\partial\Lambda=\{v\in V {\backslash \Lambda}: \exists w\in\Lambda, v\sim w\}$.
We also denote by $\mathcal{F}_\Lambda$ the 
$\sigma$-algebra generated by the variables {$\sigma_{\Lambda}=(\sigma_v)_{v\in \Lambda}$}, or equivalently by the cylinders denoted $\{\sigma_{{\Lambda}}=\omega_\Lambda\}$.
 Another important $\sigma$-algebra will be the tail $\sigma$-algebra of asymptotic events $\mathcal{F}_\infty= \cap_{\Lambda \Subset V} \mathcal{F}_{\Lambda^c}$ {where 
 the intersection runs over finite subsets.}
\\

We introduce an interaction potential $\Phi$ and consider equilibrium states to be Gibbs measures built with the DLR framework, see e.g. \cite{HOG}: they are the probability measures $\mu$ consistent with the Gibbsian specification $\gamma^\Phi$ in the sense that a version of their conditional probabilities w.r.t. the outside of any finite set $\Lambda$ of the tree is provided by the {\it DLR equations},
\begin{align}\label{DLR}
\forall \Lambda \Subset V,\; \forall \omega_\Lambda \in \Omega_\Lambda,\; \mu[\sigma_\Lambda=\omega_\Lambda \mid \mathcal{F}_{\Lambda^c}](\cdot) = \gamma_\Lambda^\Phi(\omega_\Lambda \mid \cdot), \; \mu-a.s.
\end{align}
where the elements of the Gibbs specification $\gamma^\Phi=\gamma^\Phi(\beta)$ are the probability kernels $\gamma_\Lambda^\Phi$ from  $\Omega_{\Lambda^c}$  
to $\mathcal{F}_\Lambda$ 
defined for all {$\Lambda \Subset V$}   as
\begin{align}\label{Gibbs}
	\gamma_\Lambda^\Phi(\omega_\Lambda \mid \omega_{\Lambda^c})=  \frac{1}{Z_\Lambda^\omega}e^{-\beta H_\Lambda(\omega_{\Lambda} \omega_{\Lambda^c})}.
\end{align}
The partition function $Z_\Lambda^\omega$ is the  normalization constant for a boundary condition $\omega_{\Lambda^c}$, at finite volume $\Lambda$, and the corresponding Hamiltonian $H_\Lambda^\Phi$ with b.c. $\omega_{\Lambda^c}$ is  provided by
\begin{align}
	H_\Lambda(\omega_\Lambda  \omega_{\Lambda^c}) = \sum_{A \cap \Lambda \neq \emptyset} \Phi_A( \omega_\Lambda  \omega_{\Lambda^c})
\end{align}
 where $\omega_\Lambda  \omega_{\Lambda^c}$ denotes the concatenation of $\omega_\Lambda$ and $\omega_{\Lambda^c}$.  
 We write $H$ for the Hamiltonian with free boundary conditions:
\begin{equation}\label{freeH}
 H_\Lambda(\omega)=\sum_{A \subset \Lambda} \Phi_A(\omega_\Lambda).
\end{equation}
and define the {\it free state} $\mu$ to be the Gibbs measure obtained by taking the infinite-volume limit with free boundary conditions.
{The limit exists in particular, when all finite volume measures are consistent, which is e.g. the case 
for the Ising model or the Potts model in zero external 
field. {In general, in a non-zero field one could think to consider sub-sequences with free 
boundary conditions, 
but it is necessary to put suitable non-deterministic 
boundary conditions generalizing the free ones, 
for which we will need the theory of boundary laws, see 
Section \ref{sec-BL} and in particular Theorem \ref{thm-BL}.} 

{Throughout the paper }we consider ferromagnetic nearest-neighbor (n.n.) potentials of the form
\begin{align}\label{uU-condition}
    \Phi_{\{v,w\}}(\omega)&=
	 \sum_{i,j=1}^{q} u_{i,j}\cdot 
\mathbf{1}_{\{(\omega_v,\omega_w)=(i,j)\}}{\cdot}{\mathbf{1}_{v \sim w}}\nonumber\\
\text{where }u_{i,i}&=0 \text{ and }  0<u:=\min_{i\neq j}u_{i,j}\leq\max_{i,j}u_{i,j}=:U
 \end{align}
 {The strict positivity of the lower bound expresses that homogeneous spin 
 configurations are energetically favored by the pair interaction, while the energies 
 of excitations 
 w.r.t.\ different ground states may depend in general on the ground state.}
The latter includes the $q$-state Potts model (the Ising model corresponding to $q=2$), for which
\begin{equation}\label{model-potts}
\Phi^{\text{Potts}}_{\{v,w\}}(\omega) = \mathbf{1}_{\omega_v \neq \omega_w}\cdot\mathbf{1}_{v \sim w}.
\end{equation}
More generally {than the Potts model}, we consider
$q$-state {\it clock models} {as an important sub-class}. {In these models the pair potential has a discrete 
rotational symmetry which means that there is a positive function $\bar u$ for which   
\begin{equation}\label{model-clock}
u_{i,j}=\bar u(|i-j|)
\end{equation}
where $|i-j|$ is the distance between $i$ and $j \in \Z_q$. As the finite-volume Gibbs measures 
with free boundary conditions are consistent measures in clock-models, they immediately yield a well-defined free state in infinite volume.
} 


Finally we extend our framework {from the free state in clock models} 
to {\it central states} (see the definition in Section \ref{sec-def-central}), {which 
are constructed as (small) deformations of a free state.}
{Here we do not assume the strict 
discrete clock-symmetry of the interaction.} 
{Observe that the simplest example of such a central state which is not a free state, is obtained for the low-temperature Ising model in a small field.}

%

%

As our proofs 
do not rely on symmetry considerations, we are \red{}indeed} able to extend them to Hamiltonians of the following type
\begin{equation}
	\begin{split}\label{potential-central}
		H(\omega)=\sum_{v\sim w}\Phi_{\{v,w\}}(\omega)+\sum_{v\in V}\Psi(\omega_v).
	\end{split}
\end{equation}
with a pair potential $\Phi_{\{v,w\}}$ fulfilling \eqref{uU-condition} uniformly in $\{v,w\}$ and a homogeneous single site potential $\Psi:\Z_q\rightarrow \R$ such that 
\begin{equation}\label{hyp-psi}
	\Vert\Psi\Vert_\infty\leq u(d-1)/8.
\end{equation}
 We moreover need a suitable assumption on the pair potential and the single-site potential imposed by the 'lazyness' condition (see Definition \ref{def-p1}).

\subsection{Choquet simplex of DLR measures}\label{sec-DLR}

We denote by $\mathcal{G}(\gamma)$ the set of DLR measures satisfying DLR equations (\ref{DLR}) for a general specification $\gamma$, see \cite{Pres76, HOG}.  {Note that by compactness of $\Omega_0$, $\mathcal{G}(\gamma) \neq \emptyset$.} This convex set has in general  a particularly interesting structure of being a {\it Choquet simplex}, that is a compact convex set possessing a subset of extremal elements ${\rm ex}\mathcal{G}(\gamma)$ such that any $\mu\in \mathcal{G}(\gamma)$ has a unique convex decomposition onto ${\rm ex}\mathcal{G}(\gamma)$. These extremal elements are mutually singular and considered to be the physical {\it states} of the system, see the discussions in \cite{HOG, FV-book, LN08}.
{Note that spatially homogeneous states may decompose into extremal but non-homogeneous 
states. It is the precisely the purpose of this work, to describe how this happens for free and central 
states on trees.}

 To briefly formalize this, denote 	$\mathcal{M}_1^+(\Omega)$ to be the set of probability measures on $(\Omega,\mathcal{F})$ and recall the tail $\sigma$-algebra of asymptotic events
$$
\mathcal{F}_\infty = \cap_{\Lambda \Subset V} \mathcal{F}_{\Lambda^c}.
$$
\begin{defn}[see \cite{HOG}]	Any $\mu \in {\rm ex}\mathcal{G}(\gamma)$ is characterized by the equivalent items:
	\begin{enumerate}
	
		\item Tail-triviality:
		$
		A \in \mathcal{F}_\infty \Longrightarrow \mu(A) \in \{0,1\}.
		$
		\item Short-range correlations
	  	\begin{equation}\label{short}
		\lim_{\Lambda \uparrow V} \sup_{B \in \mathcal{F}_{\Lambda^c}} \big| \mu(A \cap B) - \mu(A) \mu(B) \big| = 0.
	 	\end{equation}
	\end{enumerate}
\end{defn}

By consistency, for any {$\Lambda \Subset V$}, the kernel $\gamma_{\Lambda}$ is  a regular version
of conditional probability of $\mu$ given
$\mathcal{F}_{\Lambda^c}$:
\begin{equation}\label{DLRcubes}
	\gamma_{\Lambda}(A |
	\cdot)=\mu[A|\mathcal{F}_{\Lambda^c}](\cdot), \;  \mu-{\rm
		a.s.},\;\forall A \in \mathcal{F}
\end{equation}
which means in particular that for any sequence $\Lambda_n \uparrow V$,
$$
\forall \mu \in \mathcal{G}(\gamma),\; \mu(\cdot)=\int_\Omega
\gamma_{\Lambda_n}( \cdot | \omega) \; d \mu(\omega).
$$
One eventually gets the starting point of the simplicial decomposition
\begin{equation}\label{startdec0}
	\forall \mu \in \mathcal{M}_1^+(\Omega),\; \mu(\cdot)=\int_\Omega
	\mu[\cdot\, | \mathcal{F}_\infty](\omega) \; d \mu(\omega)
\end{equation}
obtained from the DLR equations by using standard Backward Martingale Limit theorem with respect to the filtration $(\mathcal{F}_{{\Lambda_n^c}})_{{n}}$. Define now, for any $\omega$, the so-called {{\it boundary condition at infinity kernel}} $\pi(\cdot|\omega)$ to be the tail-measurable kernel
\begin{align}\label{pi-kernel}
	\pi(\cdot|\omega)=
	\lim_{\Lambda\uparrow\mathcal T^d}\gamma_\Lambda(\cdot|\omega)
\end{align}
if the limits exists for all cylinder {events} (generating $\mathcal{F}$), 
and $\pi(\cdot|\omega):=${$\nu$} to be a {fixed arbitrary probability measure else}.

\begin{thm}[see \cite{HOG}]\label{thm-pi-kernel-extremal}
	The limit \eqref{pi-kernel} exists and is an extremal Gibbs measure for $\mu$-almost every $\omega$.
\end{thm}
 Equation  \eqref{startdec0} thus yields that 
for any $\mu$ satisfying \eqref{DLR}, 
\begin{align}\label{limDLR}
	\mu(\cdot)=\int \pi(\cdot|\omega) d\mu(\omega).
\end{align}

{To work} with probability measures on spaces $\mathcal{M}$ of (probability) measures, we first endow such spaces
with a canonical measurable structure. For any subset of
probability measures $\mathcal{M} \subset
\mathcal{M}_1^+(\Omega)$, the natural way to do so is
to evaluate any $\mu \in \mathcal{M}$ via the numbers $\big
\{\mu(A), \; A \in \mathcal{F} \big \}$. One introduces then the
{\it evaluation maps} on $\mathcal{M}$ defined for all $A \in
\mathcal{F}$ by
\begin{equation}\label{eval}
	e_A : \mathcal{M} \longrightarrow [0,1]; \mu \longmapsto
	e_A(\mu)=\mu(A).
\end{equation}

The {\it evaluation} $\sigma$-{\it algebra} $e(\mathcal{M})$ is
then the $\sigma$-algebra generated by the sets $\{e_A \leq c\}$ for all $A \in \mathcal{F}, \; 0 \leq c \leq 1$, or equivalently the smallest $\sigma$-algebra on $\mathcal{M}$ that makes
these evaluation maps measurable.
For any bounded measurable function $f \in \mathcal{F}$,
the map $$e_f : \mathcal{M} \longrightarrow \mathbb{R}; \mu \longmapsto e_f(\mu):=\mu[f]$$ is then $e(\mathcal{M})$-measurable. {It is direct by standard measure-theoretic techniques that all the events and maps we consider in this work are measurable in this sense.}

\bigskip

\begin{thm}[\cite{Dy, HOG,FV-book,LN08}] \label{pipoint}
	Assume that  $\mathcal{G}(\gamma)\neq \emptyset$. Then
	$\mathcal{G}(\gamma)$ is a convex subset of
	$\mathcal{M}_1^+(\Omega)$
	whose extreme boundary is denoted ${\rm ex} \mathcal{G}(\gamma)$, and satisfies  the following properties:
		\begin{equation}\label{Choquet}
			\forall \mu  \in \mathcal{G}(\gamma),\; \mu = \int_{{\rm ex}\mathcal{G}(\gamma)} \; \nu \cdot \alpha_\mu(d
			\nu)
		\end{equation}
		where $\alpha_\mu \in \mathcal{M}_1^+\Big({\rm
			ex}\mathcal{G}(\gamma),e({\rm ex} \mathcal{G}(\gamma))\Big)$ is
		defined for all $M \in e({\rm ex} \mathcal{G}(\gamma))$ by
		
		\begin{equation} \label{weights}
			\alpha_\mu (M) = \mu \Big[ \big \{ \omega \in \Omega: \pi(\cdot | \omega) 
			\in M\big \} \Big]. 
		\end{equation}

\end{thm}


{Our goal in this work is to rigorously establish that for the low temperature 
free state $\mu$ {(and for central states)}
the measure $\alpha_\mu$ has no atom on ${\rm ex} \mathcal{G}(\gamma) $.}

\subsection{Tree-indexed Markov chains}

A probability measure $\mu$ on $(\Z_q)^V$
 is a homogeneous tree-indexed Markov chain if and only if it allows the following iterative construction:
\begin{enumerate}
\item Sample $\sigma_0$ at (arbitrary) root $0$ according to single-site marginal of $\mu$. 
\item Sample $\sigma_w$ via some transition matrix $P(\omega_v,\omega_w)$ 
from inside to outside.
\end{enumerate} 
The abstract general definition of tree-indexed Markov chain which does not 
assume any invariance under any tree-automorphism requires that  
 \begin{equation}\label{transi}
 	\mu(\sigma_w=\cdot |\F_{\text{past of }(v,w)})=\mu(\sigma_w=\cdot |\F_v)
 	\end{equation}
holds for all oriented edges $(v,w)$, and the past of an oriented edge is the set 
of vertices which are closer to $v$ than to $w$.    

We note that in the trivial case $d=1$ (that is excluded for the most part in the analysis 
of this paper) $\mu$ is a homogeneous Markov chain if it is shift-invariant and 
\textit{reversible}. This is clear, as time-reversal corresponds to a reflection of $\Z$. 
We note the important background theorem: 

\begin{thm}[\cite{HOG} Theorem 12.6]\emph{}\\ If $\mu$ is an extremal Gibbs measure then $\mu$ is a tree-indexed 
Markov chain. 
\end{thm}
The converse is not true, as the famous example of the free state 
for the Ising model in zero external field shows {(see Higuchi \cite{hig77}, Remark p342)}.

In the whole article all we need to suppose is  
that $\mu$ is a homogeneous tree-indexed Markov chain for a Hamiltonian in the 
class  \eqref{potential-central}, {with a certain lazyness property in the following sense.} 

\begin{defn}[Lazyness parameter]\label{def-p1}
	Let $\mu$ be a homogeneous tree-indexed Markov chain.
	Denote by $P$ its transition matrix (which is then homogeneous, too) and define the corresponding {maximal jump probability} to be  
	\begin{align}
	p_1=p_1(\mu):=\max_{i\in\mathbb Z_q }\sum_{j\neq i}P(i,j).
	\end{align}
	Equivalently, for all $v\sim 0$ we have  
	\begin{align}\label{def-p1-ter}
	p_1=\max_{i\in\mathbb Z_q }\mu(\sigma_v\neq i|\sigma_0=i ).
	\end{align}
\end{defn}

In the course of the proof of our main theorem in Section \ref{sec-results}, which asserts the continuity of the extremal 
decomposition measure of a Markov chain $\mu$, 
we ask for sufficient smallness 
of the quantity $p_1(\mu)$.
This requirement means that all states are sufficiently lazy, i.e.\ with large probability 
the chain stays in each of the states. \\

Below we will see that the examples  of {\it central states} 
obtained by small perturbations of {low temperature} free states of 
clock models keep the {Markov chain property} and the small $p_1$ property. 
{The reason
is that $p_1$ deforms smoothly under perturbations, and so 
lazyness carries over from the unperturbed model.}

\section{Results}\label{sec-results}

\subsection{Reconstruction bounds for central states} 

{We have the following reconstruction statement, for
the {central states} in any model {with Hamiltonian} of the class \eqref{potential-central}. }

{\begin{thm}[]\label{thm-reconstruction} 
Consider any central state $\mu$ in the class of models {with Hamiltonian of the form} \eqref{potential-central}, fulfilling the bounds \eqref{uU-condition} and \eqref{hyp-psi}. Then,
 there exists $\beta_0>0$ large enough such that for any $\beta>\beta_0$, for any $a\in\Z_q$,
	\begin{align}\label{eq-reconstruction}
	\mu\Bigl( \lbrace \omega:  
	\pi(\sigma_0=a|\omega)\leq 1- \epsilon_1(\beta) \rbrace
		\Bigl| \sigma_0=a  \Bigr) 
		\leq \epsilon_2 (p_1)
	\end{align}
	where $\epsilon_1(\beta)=\epsilon_1(\beta,u,U,d)\downarrow0$ as $\beta\uparrow\infty$ and $\epsilon_2(p_1)=\epsilon_2 (p_1,u,U,d)\downarrow0$ as $p_1\downarrow0$.\\
	In particular, if $\mu$ is the free state of a clock model, the reconstruction bound
	\begin{align}\label{eq-reconstruction-small}
				\int d\mu(\omega|	\sigma_0=a)\pi(\sigma_0=a|\omega)>\frac1q=\mu(\sigma_0=a)
	\end{align}
	holds at large enough $\beta$.
\end{thm}

{From a signal recontruction point of view} (see \cite{Mossel2001}), the statement says that a signal $a$ which is sent from 
the origin to infinity through noisy canals can be restored by the best tail-measurable 
predictor $\pi(\sigma_0=a|\omega)$ up to thermal fluctuations $\epsilon_1(\beta)$ 
up to an error probability $\epsilon_2(p_1)$. 

%




The proof is a direct consequence 
of the Key Lemma \ref{key-lemma} proved in Section \ref{sec-proofs-1} which relies on a good site/bad site decomposition for boundary conditions $\omega$, together with Peierls bounds under disorder.   

\subsection{Positivity of the Edwards-Anderson parameter and non-extremality}

The so-called Edwards-Anderson parameter is usually defined in spin glass models as a quantity measuring the degree of randomness of the (random) spin magnetisation at the origin (see e.g.\ \cite{CCST}). In our context, the analogue {should} be a quantity measuring the degree of randomness of the probability vector  
$(\pi(\sigma_0=a|\omega))_{a\in \Z_q}
$, when
the boundary condition at infinity $\omega$ is distributed according to $\mu$. 
 Let us thus define the following quantity.
 \begin{defn}[Central state Edwards-Anderson parameter]\label{def-qEA}
	\begin{equation}
		\begin{split}
		q_{\rm EA}^\mu
	&:=\frac1q\sum_{a\in \Z_q}{\rm Var}_\mu\bigl({\pi( {\sigma_0=a} | \cdot})\bigr).\cr
	\end{split}
	\end{equation}
 \end{defn}
Note that when $\mu$ is the free state of a clock model, 
by symmetry the above definition boils down to
\begin{equation}\label{qEA-clock}
	\begin{split}
	q_{\rm EA}^{\text{Clock}}&:=
	{\rm Var}_\mu\bigl({\pi( {\sigma_0=a} | \cdot})\\
	&=
	\mu({\pi( {\sigma_0=a} | \cdot})^2)-\frac1{q^2}
\end{split}
	\end{equation}
for any fixed $a\in \Z_q$, whereas for the free state of the Ising model in zero field, Definition \ref{def-qEA} writes
	\begin{equation}\begin{split}
	&q_{\rm EA}^{\text{Ising}}=\frac14
	{\rm Var}\bigl({\pi( {\sigma_0} | \cdot)}\bigr).
	\end{split}
	\end{equation}
Clearly, having $q_{\rm EA}>0$ implies that there exists some $a\in\Z_q$ such that ${\rm Var}_\mu(\pi(\sigma_0=a|\cdot))>0$ and thus, recalling \eqref{startdec0}, the tail-measurable random variable $\mu(\sigma_0=a|{\cal{F}_\infty})(\cdot)$ is not $\mu$-a.s. constant.  Thus $\mu$ cannot be tail-trivial.
We have the following quantitative lower bound.

	\begin{thm}[]\label{thm-EA} Consider any central state $\mu$ in the class of models {with Hamiltonian of the form} \eqref{potential-central} fulfilling the bounds \eqref{uU-condition} and \eqref{hyp-psi}. Then,
		there exists $\beta_0>0$ large enough such that for any $\beta>\beta_0$,
	\begin{align}
		q_{\rm EA}^\mu\geq 
		(1-\epsilon_1(\beta))^2 (1- \epsilon_2(p_1))
			-\max_{a\in \Z_q}\mu(\sigma_0=a)
	\end{align}
	where $\epsilon_1(\beta)\downarrow0$ as $\beta\uparrow\infty$ and
	$\epsilon_2 (p_1)\downarrow0$ as $p_1\downarrow0$ are the same quantities as in Theorem \ref{thm-reconstruction}.
	In particular, for large enough $\beta$ and small enough $p_1=p_1(\mu)$, 
	$$q_{\rm EA}^\mu>0$$ 
	and thus $\mu$ is not extremal.
\end{thm}

Note that for the free state of clock models 
$\max_{a\in \Z_q}\mu(\sigma_0=a)$ is equal to $1/q$ by symmetry. More generally, for a central state this term is close to $1/q$ (see \eqref{close-to-free} and Definition \ref{def-central} below). 

Theorem \ref{thm-EA} follows from Theorem \ref{thm-reconstruction}. The proof is elementary and can be found in Section \ref{proof-thm-EA}.

\subsection{{Continuity of the extremal 
decomposition of the free state} }
We prove in the following theorem that uncountably many extremal Gibbs measures enter the decomposition of the central state at low enough temperature.


\begin{thm}[Almost-sure singularity of extremals]\label{thm-sing-ex} 
	Consider any central state $\mu$ in the class of models {with Hamiltonian of the form} \eqref{potential-central} fulfilling the bounds \eqref{uU-condition} and \eqref{hyp-psi}. Then,
	there exist $\beta_0>0$ large enough and $p_1^0=p_1^0(\beta,u,U,d)>0$ small enough such that for any $\beta>\beta_0$, $p_1(\mu)<p_1^0$ and for $\mu\otimes\mu$-a.e.\ pair $(\omega,\omega')$ the extremal Gibbs measures 
$\pi(\cdot|\omega)$ and $\pi(\cdot|\omega')$ are singular with respect to each other, i.e.
\begin{align}
	\mu \otimes \mu\big(\{(\omega,\omega')\in \Omega\times \Omega : \pi(\cdot |\omega) \perp \pi(\cdot | \omega')\}\big)=1.
\end{align}  
\end{thm}


\begin{cor}\label{cor-uncountably} {The decomposition measure $\alpha_\mu$ 
has no atoms,  i.e. $\alpha_\mu(\{\nu\})=0$ for all $\nu\in{\rm ex}\mathcal G(\gamma)$. In particular there are 
uncountably many extremal states which enter in the extremal decomposition of $\mu$. }
%
\end{cor}
\begin{proof} {Suppose the opposite, namely 
that $\mu(\pi(\cdot|\omega)=\mu_0)>0$, for some atom $\mu_0$. Then
\begin{align*}\label{sing}
	&\mu \otimes \mu
	(\{(\omega,\omega') : \pi(\cdot |\omega) \perp \pi(\cdot | \omega')\})\cr
&	=
	1-\mu \otimes \mu
	(\{(\omega,\omega') : \pi(\cdot |\omega) = \pi(\cdot | \omega')\})\\
	&\leq 1-\mu(\pi(\cdot|\omega)=\mu_0)^2
	<1.
\end{align*}
which is a contradiction.}
\end{proof}

The idea to prove almost sure singularity of typical extremals taken from the product measure 
is to produce a tail-measurable order parameter, which carries enough information 
to distinguish to two different typical extremal Gibbs measures. 
Indeed, as the infinite-volume kernel $\pi(\cdot|\omega)$ is supported on the extremals, 
and extremals are uniquely described by their restriction to the tail-sigma 
algebra, it suffices to find a tail-measurable observable $\phi$ on 
which the expectations 
$\pi(\phi|\omega)$ and $\pi(\phi|\omega')$ differ.  

For this purpose we construct $\phi$ by looking at empirical 
sequences of overlaps of the spin variables $\sigma$ with $\omega$,  
and show that its expectation becomes big in $\pi(\phi|\omega)$ on the one hand, 
and small in $\pi(\phi|\omega')$ on the other hand, 
for typical choices of $(\omega,\omega')$. 
{Theorem \ref{thm-sing-ex} will thus follow from a control of so-called {\it branch overlaps} defined in the section below.
}

\subsection{Concentration of branch-overlap for typical extremals}
\label{sec-branch-overlap-sym}

Let $n\in\N$ and ${\rm r}=(r_1,r_2,\ldots)$ be an increasing sequence of positive integers. Let
		\begin{equation}
		\begin{split}
		\label{lambda-n-r}
		&{\Lambda_n=\Lambda_n^{\rm r}}=\{v_1,v_2,\dots,v_{n^2}\}\subset V
		\end{split}
		\end{equation}
	so that $|\Lambda_n|=n^2$, and the vertices $v_i$ are chosen along a branch of the tree, 
	in such a way that their spacing in graph distance is given by the sequence ${\rm r}$, i.e. $|v_{i+1}-v_i|=:r_i$ for all $i\in\{1,\ldots, n^2\}$. 

	
\begin{defn}[Thinned branch overlaps] \label{def-branch-overlap}
	Define the tail-measurable function, called
	{\it thinned branch-overlap}, measuring how much the configuration $\sigma$ matches with $\omega$ on the increasing sparse volumes $\Lambda_n$, as
	\begin{equation}\begin{split}
	&\underline\phi^\omega=\liminf_{n\uparrow\infty}\frac{1}{|\Lambda^{\rm r}_n|}\sum_{v\in \Lambda^{\rm r}_n}1_{\sigma_v=\omega_v}.\cr
	\end{split}
	\end{equation}
	For any fixed configuration $\omega$, this is a tail-measurable observable 
	w.r.t.\ the dependence on the spins $\sigma$, which takes 
	values on the interval $[0,1]$. Analogous  tail-measurability also holds w.r.t.\ to the parameter $\omega$. 
\end{defn}

Note that there  is thinning  
	in two ways: the volume becomes increasingly sparse, and the {liminf}
	is taken along volumes of $n^2$ sites. 
The following theorem can be viewed as a quantitative statement of the glassiness of the states $\pi(\cdot|\omega)$, for $\mu$-almost every $\omega$: it describes quantitatively how much typical configurations $\sigma$ sampled from $\pi(\cdot|\omega)$ are $\omega$-like.

\begin{thm}[Branch overlap]\label{thm-branch-overlap} 
	Consider any central state $\mu$ in the class of models {with Hamiltonian of the form} \eqref{potential-central} fulfilling the bounds \eqref{uU-condition} and \eqref{hyp-psi}. Then
there are two functions $\epsilon_1(\beta)\downarrow 0$ as $\beta\uparrow \infty$ and $\epsilon_2(p_1)\downarrow0$ as $p_1\downarrow0$ such that for sparse enough sets $\Lambda_n=\Lambda_n^{\rm r}$ 
the following holds:\\
For $\mu$-a.e. $\omega$, 
the thinned branch-overlap $\underline\phi^\omega$
{is $\pi(\cdot |\omega)$-a.s.\  lower bounded by}
\begin{equation}\begin{split}
\label{concentration-bound}
\underline\phi^\omega=
&\liminf_{n\uparrow\infty}\frac{1}{|\Lambda_n|}\sum_{v\in \Lambda_n}1_{\sigma_v=\omega_v}
\geq 1-\epsilon_1(\beta)-\epsilon_2(p_1).
\end{split}
\end{equation}
\end{thm}
The quantities $\epsilon_1(\beta)$ and $\epsilon_2(p_1)$ are the same quantities as in Theorem \ref{thm-reconstruction}.
The proof can be found in Section \ref{proof-thm-branch-overlap}.
{From Theorem \ref{thm-branch-overlap}, we deduce that the tail-measurable observable $\underline\phi^\omega$ has different expectations under $\pi(\cdot|\omega)$ and $\pi(\cdot|\omega')$ if $\omega\neq\omega'$, through the following corollary. }

\begin{cor} \label{cor-glassy}
Let $\beta$ be large enough and $p_1$ be small enough such that
\begin{equation}
\label{epsilon-small}
\epsilon_1(\beta)+\epsilon_2(p_1)
<\frac{1}{2}(1-{\sum_{a\in \Z_q}\mu(\sigma_0=a)^2}).
\end{equation}
Then there is a sequence 
of integers ${\rm r}=(r_i)_{i\in \N}$ such that for the correspondingly 
defined tail-measurable observable $\underline\phi^\omega$ we have the strict 
inequality 
	\begin{equation}
\begin{split}
\label{eq-glassy}
&\pi(\underline\phi^\omega|\omega)>\pi(
\underline\phi^\omega|\omega')
\end{split}
\end{equation}
for $\mu\otimes\mu$-a.e. $(\omega,\omega')$. 
\end{cor}
The quantities $\epsilon_1(\beta)$ and $\epsilon_2(p_1)$ are the same quantities as in Theorem \ref{thm-reconstruction}.
{The proof can be found in Section \ref{proof-cor-glassy}. 
From Corollary \ref{cor-glassy} the statement of Theorem \ref{thm-sing-ex} is immediate. 
}

\subsection{Structure of the proofs}

To prove the concentration bound of Theorem \ref{thm-branch-overlap}
it is useful to adopt a quenched-disordered systems view (we refer to \cite{B2006} for a general introduction to statistical mechanics of disordered systems). A configuration $\omega$ drawn from the free/central state $\mu$ corresponds to a boundary condition at infinity, but also plays the role of disorder.
The intuition {behind} is that the extremal measure $\pi(\cdot|\omega)$, which is in general inhomogeneous 
and plays the role of a quenched disordered state,
mostly resembles $\omega$ locally, and for $\mu$-typical $\omega$.

To make this precise, we will need contours, as introduced in \cite{CoKuLe22},
and Peierls bounds relative to the reference configuration $\omega$. 
Typically $\omega$ will contain a small density of broken bonds (along which $\omega$ changes). However 
the latter will not be uniformly sparse, as needed in \cite{CoKuLe22} and \cite{GRS12} to ensure the excess energy estimate leading to Peierls bounds. 
To treat the rare but arbitrarily large regions where the Peierls bound locally fails, we introduce a notion of "bad site" (see Definition \ref{def-bad}), which is an essential tool.   
Our good/bad site decomposition is somewhat reminiscent to that invented by Chayes-Chayes-Fr\"ohlich in \cite{CCF85}, to treat lattice Ising models with i.i.d.\ random bonds, which are mostly but not strictly ferromagnetic. 
However, we work on the tree, and {in a regime where} the "disorder-measure" $\mu$ is not tail-trivial, it is far from an i.i.d.\ disorder measure, making things more intricate.

Nevertheless, {as we shall see}, there is one-dimensional correlation decay, conditionally on the state of the root, along a branch of the tree. 
This will be exploited to prove exponential decorrelation of bad sites in their distance, see Lemma \ref{lem-decorr},
leading to concentration of thinned branch overlaps around their means, under the "quenched measures" $\pi(\cdot |\omega)$, for $\mu$-almost every $\omega$, see Lemma~\ref{lem-overlap}.

\newpage
\section{Proofs}\label{sec-proofs}

We first recall some tools about tree-indexed Markov chains and boundary laws.

\subsection{Markov chains on trees and boundary laws.}\label{sec-BL} Being {nearest neighbor} (n.n.), our class of models \eqref{potential-central} lead to Gibbs measures that are {(spatially)} Markov fields. {As we work on trees, there is an important class of Gibbs measures 
including the extremal Gibbs measures but not equal to them, 
which has  a direct transcription in terms of Markov chains on trees. These are described 
via so-called boundary laws introduced by Zachary \cite{zach83}. Boundary laws 
are (non-normalized) positive measures which are invariant under an interaction-dependent
non-linear map along the tree. Moreover they are closely related (but not equal) to 
the invariant single-site probability measures for the associated 
one-step Markov chain transition matrix, see below, and 
see Georgii \cite[Chapter 12]{HOG} for details.}  \\

In our case of a homogeneous n.n. interaction on the tree, 
the specification \eqref{Gibbs} is equivalently described by a positive 
\textit{transfer matrix (or transfer operator)} $Q:\Z_q\times 
\Z_q\mapsto (0,\infty)$ via the prescription 
\begin{equation*} \label{Def: Gibbs specification}
	\gamma_\Lambda(\omega_\Lambda \mid \omega_{\Lambda^c})=\frac1{Z^\omega_\Lambda} \prod_{{ \{v,w\} \cap \Lambda \neq \emptyset}\atop{v\sim w}}Q(\omega_v,\omega_w).
\end{equation*}
where, writing $b=\{v,w\}$, {with pair potential $\Phi$, and single-site potential $\Psi$ as in \eqref{potential-central}, and}
$$
Q(\omega_v,\omega_w)=:Q_b(\omega)=
e^{-\beta(\Phi_b(\omega) +( \Psi(\omega_v)+\Psi(\omega_w))/(d+1))}.
$$

\subsubsection{Boundary laws}

\begin{defn} A {boundary law} $\lambda$ for a transfer matrix $Q$ is a family of row vectors { $\lambda_{vw} \in \; ]0,\infty[^{\mathbb{Z}_q}$} which satisfy, {for all {oriented pairs 
of nearest neighbors }$v,w \in V$},  the following consistency equation: there exists some $c_{vw} >0$ such that for all $i \in \mathbb{Z}_q$,
	
	\begin{equation*} \label{eq: perGenBL}
		\lambda_{vw}(i)=c_{vw}
		\prod_{z \in \partial \{v\} \setminus \{w\}} \sum_{j\in\Z_q}
		Q(i,j)\lambda_{zv}(j).
	\end{equation*}  
\end{defn}

\begin{thm}[\cite{zach83}]\label{thm-BL} There is a one-to-one relation between 
	Gibbs measures $\mu$ which are also tree-indexed Markov chains and {boundary laws} $\lambda$. \\  
	{A} Gibbs measure $\mu$ is then described via its finite-volume marginals in any $\Lambda\Subset V$	in the form 
	\begin{equation}\label{BoundMC}
		\mu(\omega_{\Lambda \cup \partial \Lambda}) = (Z^\omega_\Lambda)^{-1} \prod_{w \in \partial \Lambda} \lambda_{w w_\Lambda}(\omega_w) \prod_{b \cap \Lambda \neq \emptyset} Q_b(\omega),
	\end{equation}
	where $w_{{\Lambda}}$ is the unique neighbor of $w\in \partial \Lambda$ 
	which lies in $\Lambda$. 
	The Markov chain transition operator is given {for all $v,w \in V$ by the stochastic matrix}
	\begin{equation}\label{transition-matrix}
		P_{vw}(\omega_v,\omega_w)=\frac{Q(\omega_v,\omega_w)\lambda_{wv}(\omega_w)}
		{\sum_{j\in\Z_q}Q(\omega_v,j)\lambda_{wv}(j)}. 
	\end{equation}
\end{thm}

We note that no homogeneity of the Gibbs measure, boundary law, and transition 
operator are assumed. It is perfectly possible and very relevant that homogeneous 
$Q$ allow for non-trivial non-homogeneous Gibbs measures $\mu$  (as for example the Blekher-Ganikhodgaev states \cite{BLG91}, which {can be seen as some analogues} of the Dobrushin states \cite{Dob72} on the tree). 
On the other hand, the theory presented above also allows for non-homogeneous 
interactions depending on the edges, but we do not consider this case in 
the present work.

{In the special case of homogeneous boundary laws $\lambda_{vw}(i)\equiv u(i)$ on regular trees of 
	degree $d$, considered here, note that} these must satisfy 
the homogeneous equation 
\begin{equation}\label{hom-BL}
	u(i)=c\big(\sum_{j\in\Z_q}
	Q(i,j)u(j)\big)^d,  
\end{equation}
which we may write in short notation as $u=c (Qu)^d$. The constant $c$ can be chosen to our convenience, a possible and often convenient choice 
is $c=1$. The single-site marginal of the measure given in \eqref{BoundMC} then becomes\footnote{where the $\ell^p$ norm of a function $f:\Z_q\to\R$ is defined by $\Vert f\Vert_p:=(\sum_{i\in\Z_q}|f(i)|^p)^{1/p}$.}
${u^\frac{d+1}{d}}/{\Vert u^\frac{d+1}{d}\Vert_1}$ and the transition matrix given in \eqref{transition-matrix} becomes
\begin{equation}\label{hom-transition-matrix}
	P^u(i,j)= \frac{Q(i,j)u(j)}{\sum_{k\in\Z_q}Q(i,k)u(k)}, \text{ for } i,j \in \Z_q.
\end{equation}

{In our non-hard core context, homogeneous boundary laws $u$ can {also} be characterized by consistent effective boundary fields $h=(h_1,\dots,h_q)$, defined as {$u_i=e^{h_i}$,  for $i \in \mathbb{Z}_q$,} that themselves satisfy a consistency mean-field equation.
	
	In the Ising case, this equation involving hyperbolic tangent possibly {has 3 independent solutions }and {gives} rise to three homogeneous Markov chains $\mu^+,\mu^-$ and the free state $\mu$ {(see the work of Spitzer \cite{Spi75}, or the one of Higuchi \cite{hig77}, who called $\mu$ the 'third Markov chain', and already noticed that it was not necessarily extremal)}. The case of the free state $\mu$ is discussed in the next section.}
	
	{For the Potts model, but also very generally for $q$-clock models
		in the case of absence of an external field, 
		this equation has a unique solution at small $\beta$ and correspondingly uniqueness of the Gibbs measure.
		For large $\beta$, the situation for general clock models in the class \eqref{model-clock} is already more interesting 
		and looks as follows (see \cite{KRK14, KR17} and \cite{AbHeKuMa2023}): 
		For each subset $A\subset\Z_q$ of the local spin-space, at sufficiently large $\beta\geq \beta_0(|A|)$ 
		there is a spatially homogeneous 
		Markov chain $\mu_{A}$ whose single-site marginals $\pi_A=\mu_{A}\circ \sigma_0^{-1}$
		are concentrated on the spin-values in $A$. Moreover, the restriction of $\pi_A$ to $A$ approximately 
		equals the equidistribution on $A$, with explicit $\beta$-dependent error bounds.  
		It is important to realize that these states contain the $q$ states with singleton-localization centers $|A|=1$, constructed in \cite{HeKu21}, but at 
		sufficiently low temperatures there are always independent 
		states with non-singleton localization centers $A$. 
		More precisely, in \cite{AbHeKuMa2023}  an explicit threshold $\eta(d,N)>0$ was constructed such that 
		the condition {$\Vert Q_0-1_{0}\Vert_{\frac{d+1}{2}}<\eta(d,|A|)$}
		implies existence 
		and concentration properties of $\mu_A$, {where the function $Q_0(i):=Q(i,0)$ describes the interaction of the clock model.}    
		It is elementary that this smallness {of the norm} is implied by {the explicit low-temperature 
		condition}
		$(q-1)^\frac{2}{d+1}e^{-\beta u}<\eta(d,|A|)$ in our class of models \eqref{model-clock}, as we shall see in the next section.
	}
	\\

	\subsubsection{Free states of clock models}
	Clock models are defined by the requirement that the transfer matrix 
	depends only on the distance in $\Z_q$ between $i$ and $j$ and hence has the form  
	$Q(i,j)=Q_0(j-i)$ with an even function $Q_0$ on $\Z_q$ which we assume 
	to be strictly positive. 
	
	In that case the homogeneous boundary law equation \eqref{hom-BL} can be written in terms 
	of the discrete convolution as
	$$u=(Q_0*u)^d.$$ 
	This in particular shows that the constant boundary law 
	$u(i)=a$ for all $i\in \Z_q$ 
	solves the equation for the non-zero value of $a$ given 
	by $a=(a\Vert Q_0\Vert_1)^d$. 
	
	The Markov chain transition operator \eqref{hom-transition-matrix} obtained for constant boundary law 
	is then the normalized transition operator 
	itself, i.e. 
	\begin{equation}\label{freeme}
		P^u(i,j)=\frac{Q_0(i-j)}{\Vert Q_0\Vert_1}
	\end{equation}
	which has as unique invariant distribution the equidistribution on $\Z_q$. The formula {(\ref{BoundMC})} for the Gibbs measure $\mu$ obtained for constant boundary law 
	then reduces to 
	\begin{equation}\label{freeGibbsy}
		\mu(\omega_{\Lambda \cup \partial \Lambda})
		= (Z^\omega_\Lambda)^{-1} \prod_{b \cap \Lambda \neq \emptyset} Q(\omega_b),
	\end{equation}
	from which we also see with \eqref{freeme} that our first 
	definition of a homogeneous tree-indexed Markov chain is satisfied.
	The r.h.s.\ of \eqref{freeGibbsy} on the other hand is the obvious formula 
	for the open boundary condition finite-volume Gibbs measure in $\Lambda$. 
	In this way we see that the free state in clock models is rediscovered 
	as the tree-indexed Markov chain which is obtained for constant boundary law.

	\subsubsection{Central states of perturbed clock models}\label{sec-def-central}
	
	We now consider more generally models 
	which are perturbations of clock models, 
	by which we mean that $$Q=Q_0+\tilde Q$$ where $\tilde Q$ is small in (some) matrix norm. 
	Such a situation arises for example in the important special case of 
	a clock model which is perturbed by additional single-site terms coming from 
	a non-trivial potential $\Psi_{\{v\}}\equiv\Psi_0$. 
	In this case the transfer operator does not describe  a clock model anymore but takes the general matrix form 
	$$Q(i,j)=a(i)Q_0(i,j)a(j),$$  
	where $a(i)=e^{-\beta\Psi(i)/(d+1)}$ is close to one for all $i$, and 
	{$Q_0(i,j)$} only depends on $|i-j|$.
	Assume without loss of generality the normalization $\Vert Q_0 \Vert_1=1$, let $x:=u^\frac{1}{d}$ 
	and write the homogeneous equation \eqref{hom-BL} as $F(x,Q)=0$ with 
	\begin{equation}\label{BL-middle-state} 
		F_i(x,Q):=x(i)-\sum_{j}Q(i,j)x(j)^d  \text{ for all } i\in\Z_q.
	\end{equation}
	Then we have 
	\begin{lemma}\label{lemma-middle} 
		Suppose that all eigenvalues of  $Q_0$ are different from $1/d$. 
		Then there is a neighborhood of $Q_0$ such that for all $Q$ in this neighborhood 
		there exists a continuously differentiable solution $Q\mapsto\bar x(Q)$ of the boundary law equation \eqref{BL-middle-state}
		for $Q$, which has the property that $\bar x(Q_0)=1$.  
	\end{lemma}

	\begin{proof}[Proof of Lemma \ref{lemma-middle}]
		It is a straightforward application of the implicit function theorem. 
		We have $F(1,Q_0)=0$, so we can find a neighborhood around $Q_0$ in matrix norm such that a perturbed 
		solution exists, 
		if $D_x F(1,Q_0)=id_{q} -d Q_0$ is invertible.
	\end{proof}
	
	Note that the eigenvalues of the matrix 
	{$Q_0$} are given by discrete Fourier transform of the vector {$(Q_0(j))_{j\in\Z_q}$} and hence directly computable. 
	We note in particular that for the $\beta$-dependent normalized form 
	$$Q_0(j)=e^{-\beta \Phi_0(j)}/\Vert e^{-\beta \Phi_0}\Vert_1$$ 
	with a potential $\Phi_0:\Z_q \mapsto [0,\infty)$ which has a 
	minimimum at $0$, and satisfies $0<u\leq \Phi_0(i)$ for $i\neq 0$, the condition of Lemma \ref{lemma-middle} is valid for large enough $\beta$. 
	More quantitatively, this is ensured by 
	\begin{equation} 
		\frac{d-1}{d+1}>(q-1)e^{-\beta u}
	\end{equation}
	which follows from the Lemma \ref{LemmaQuant} below.

	\begin{defn}\label{def-central} {Consider the continuously differentiable solution $\bar{x}$ of Lemma \ref{lemma-middle}}. We call {central state} the Markov chain 
		Gibbs state associated to the pair $(\bar x(Q),Q)$ solving \eqref{BL-middle-state}. 
	\end{defn}
	{As a consequence of Lemma \ref{lemma-middle},}
	the corresponding single-site marginal
	$$Q\mapsto \pi_Q=\frac{\bar x(Q)^{d+1}}{\Vert \bar x(Q)^{d+1}{\Vert_1}}$$ 
	and transition matrix 
	$$Q\mapsto P_Q(i,j)=\frac{Q(i,j)\bar x(Q)^{d}_j}{\sum_k Q(i,k)\bar x(Q)^{d}_k}$$
	are continuously differentiable perturbations of the values
	{$\pi_{Q_0}=\frac{1}{q}1$ and $P_{Q_0}=Q_0/{\Vert Q_0\Vert_1}$ of the free state of}
	the reference clock model $Q_0$ (where we denote by $1$ the vector with all coefficients equal to 1).

	Moreover, in this normalization, the perturbed solution is 
	close to the free  state of the clock model as 
	\begin{equation} \label{close-to-free}
		\bar x(Q)=1+ (id_{q} -d Q_0)^{-1}(Q-Q_0)1 + o(\Vert Q-Q_0\Vert)
	\end{equation}
	in any matrix norm $\Vert Q-Q_0\Vert$. \\

	We now give lower bounds on the eigenvalues of the transition 
	operator of a clock model, {which are larger than $1/d$ for $\beta$ large enough, ensuring the condition in Lemma \ref{lemma-middle}.} 
	
	\begin{lemma}\label{LemmaQuant} Consider a clock model {with Hamiltonian} in the class  \eqref{model-clock} fulfilling the bound \eqref{uU-condition}. We have the lower bounds on the eigenvalues $\lambda_j$ of 
		the transition operator $Q_0=\frac{e^{-\beta \Phi}}{\Vert e^{-\beta \Phi} \Vert_1}$ of 
		the form 
		\begin{equation*} 
			\lambda_j \geq \frac{1-(q-1)e^{-\beta u}}{1+(q-1)e^{-\beta u}}.\end{equation*}
	\end{lemma}
	\proof Clearly $\Vert e^{-\beta \Phi}  \Vert_1\leq 1+(q-1)e^{-\beta u}$. 
	To estimate the spectral radius of the symmetric matrix 
	$e^{-\beta \Phi}-id_q$, observing that for $v$ with $\Vert v \Vert_2=1$ 
	we have 
	\begin{equation} 
		\begin{split}
			&|\langle v, (e^{-\beta \Phi}-id_q)v\rangle |\leq e^{-\beta u} \sum_{i\neq j}|v_i| |v_j|\cr
			&= e^{-\beta u}( (\sum_{i}|v_i|)^2 -1)\leq e^{-\beta u} (q-1).
		\end{split}
	\end{equation}
	This implies that the eigenvalues of $e^{-\beta \Phi} $ are bounded below 
	by $1-e^{-\beta u} (q-1)$, which  proves the lemma. 
	\endproof

	\subsubsection{Lazyness assumption}\label{sec-small-p1}
	
	Recall the lazyness parameter $p_1(\mu)$ in Definition \ref{def-p1}.
	For our main Theorem \ref{thm-branch-overlap} to be meaningful 
	we need sufficient smallness of the quantity $p_1(\mu)$.

	From the above follows that the{ \it central states} obtained by small perturbations 
	of clock models have an associated transition matrix which keep the small $p_1$ property. Indeed, 
	assume again the normalization of $Q_0$ such that  $\Vert Q_0 \Vert_1=1$, and $Q_0(0)=1$. 
	Suppose that all eigenvalues of  $Q_0$ are different from $1/d $, and let 
	$\mu_Q$ denotes the central state, defined in a sufficiently small neighborhood of $Q_0$.  
	Then  $Q\mapsto p_1(\mu_Q)$ is a continuously differentiable function.  
	In particular, whenever the reference clock model $Q_0$ has small  
	\begin{align}\label{def-p1-bis}
		p_1(\mu_{Q_0})=\frac{\sum_{i\neq 0} Q_0(i)}{\Vert Q_0 \Vert_1}
		\leq(q-1)e^{-\beta u},
	\end{align}
	this smallness carries over to $p_1(Q)$ for $Q$ 
	in a sufficiently small neighborhood around $Q_0$. Specifically for the Potts model 
	as a reference clock model we have 
	\begin{equation}\label{p1-free-clock}
		p_1(Q_0)=\frac{q-1}{\e^{\beta}+q-1}	
	\end{equation}
	which tends to 0 as $\beta$ tends to infinity.

\subsection{Contours and Excess Energy Lemma}

\begin{defn}[Contour with respect to a fixed configuration $\omega^0$]\emph{}\\
	Let $\omega^0\in \Omega_0^V$ be a fixed reference 
	configuration. 
	A {\it contour} for the spin configuration $\omega\in \Omega_0^V$ relative to $\omega^0$
	is a pair $$\bar \gamma=(\gamma, \omega_{\gamma})$$ where the support {
	$\gamma \ {\subset} \ \{v\in V: \omega_v\neq \omega_v^0\}$ is 
	a connected component of the set of incorrect points for 
	$\omega$ (with respect to $\omega^0$), and {$\omega_\gamma=(\omega_v)_{v\in\gamma}$}.}
	\end{defn}

{
Given a fixed reference configuration $\omega^0$, we will work with a contour representation of the spin partition function 
in a finite volume $\Lambda$ with boundary condition equal to $\omega^{0}$, which reads 
\begin{equation}
\begin{split}
Z_{\Lambda}^{\omega^0}={\sum_{n\in\N}}\sum_{\bar \gamma_1, \dots, \bar \gamma_n}
\prod_{i=1}^n \rho[\omega^0](\bar\gamma_i)
\end{split}
\end{equation}
where the sum is over pairwise compatible contours $\bar\gamma$ with activities 
$$
\rho[\omega^0](\bar\gamma)=
\exp({-\beta(H_{\gamma\cup \partial \gamma}(\omega_{\gamma\cup \partial \gamma})- 
H_{\gamma\cup \partial \gamma}(\omega^{0}_{\gamma\cup \partial \gamma}))})
$$
given in terms of the excess energy of $\omega$ with respect to $\omega^0$.  
We write $\bar\gamma_1 \sim \bar\gamma_2$ whenever $\bar\gamma_1$ and 
$\bar\gamma_2$ are compatible. 
This representation together with the tree structure of our graph (where contours have no interior) imply the following Peierls bound:

\begin{lemma}[Peierls bound]\label{lem-Peierls} {Consider  $\omega^0 \in \Omega$ and a volume $\Lambda \Subset V$. Then}
	\begin{align}
		\mu^{\omega^0}_\Lambda(\sigma_v\neq\omega^0_v)
		&=\sum_{\bar\gamma:\gamma\ni v}^{\Lambda}
		\rho[\omega^0](\bar\gamma) \cdot
	\frac{{\sum_n}\sum^{\Lambda}_{\bar\gamma_1,\dots, \bar\gamma_n \sim \bar\gamma}\prod_{i=1}^n \rho[\omega^0](\bar\gamma_i)}{
	{\sum_n}\sum^{\Lambda}_{\bar\gamma_1,\dots, \bar\gamma_n}\prod_{i=1}^n 
	\rho[\omega^0](\bar\gamma_i)}\nonumber\\
	&\leq \sum_{\bar\gamma:\gamma\ni v}^{\Lambda}
		\rho[\omega^0](\bar\gamma).
	\end{align}
	where the sum is over contours $\bar\gamma$ in the finite volume $\Lambda$.
\end{lemma}
\begin{defn}[Set of broken bonds]\emph{}\\
We define the set of broken bonds of the configuration $\omega\in\Omega$ by
$$D(\omega)=\{(v,w)\in E:\omega_v\neq\omega_w\}$$
and denote the set of edges attached to $\gamma\subset V$ by
\begin{equation}
\begin{split}
E(\gamma)&=\{\{x,y\}\in E,  \{x,y\} \cap \gamma \neq \emptyset  \}.
\end{split}
\end{equation}
\end{defn}

{
\begin{lemma}[Entropy bound, {[\cite{GRS12},Lemma 6]}]\label{lem-entropy}
	We have the following upper bound on the number of connected subsets of vertices of $\mathcal T^d=(V,E)$:
	$$\#\{\gamma\subset V : \gamma \text{ connected ,} \gamma\ni 0,|\gamma|=\ell\}\leq(d+1)^{2(\ell-1)}.$$
\end{lemma}
}

\begin{lemma}[Excess Energy Lemma {[\cite{CoKuLe22}, Lemma 2]}]
	\label{lem-EEL}\emph{}\\
	Let $\mu$ be a central state of any model {with Hamiltonian} in the class \eqref{potential-central} fulfilling the bounds \eqref{uU-condition} and \eqref{hyp-psi}.
	Let $\bar\gamma=(\omega'_{\gamma},\gamma)$ 
	be a contour relative to the fixed configuration 
	$\omega$. 
	Denote $\omega'=(\omega'_{\gamma}\omega_{\gamma^c})$ 
	the corresponding excited spin configuration. 
	Then the excess energy satisfies
	\begin{equation}
	\begin{split}	
	\label{XSfinite}
	H(\omega')&- H(\omega) 
	\geq \\
	&(d-1)u |\gamma|-(U+u) | D(\omega)\cap E(\gamma) | 
	+\sum_{v\in\gamma} (\Psi(\omega'_v)-\Psi(\omega_v)).
	\end{split}
	\end{equation}
	In particular, if there exist $\delta\geq0$ and $\epsilon>0$ such that $| D(\omega)\cap E(\gamma) | \leq\delta|\gamma|$ and $\Vert\Psi\Vert_\infty\leq\epsilon/2$ then
	\begin{equation}	\label{EEL-constant}
		H(\omega')- H(\omega) 
		\geq |\gamma|\Big((d-1)u-(U+u)\delta-\epsilon\Big).
	\end{equation}
\end{lemma}
}
Observe that in \eqref{EEL-constant}, if $\epsilon=0$ and $\delta<\frac{(d-1)u}{U+u}$ then the energy cost of the contour $\bar\gamma$ is positive and proportional to its volume $|\gamma|$. This fact was crucially used in our paper \cite{CoKuLe22} to exhibit low temperature local ground states which give rise to a wide family of inhomogenous extremal states. We thus introduce the following definition of bad contours.

\begin{defn}[Bad events] \label{def-bad} 
	Consider any model in the class \eqref{potential-central}. Let 
\begin{equation}\label{def-delta0}
	\delta_0:=\frac12\cdot\frac{(d-1)u}{(u+U)}
\end{equation}
where $u,U\in\R^+$ are the bounds on energy costs defined in \eqref{uU-condition} and $d$ is the branching number of the tree. \\
Denote by $B_v$ the bad event 
that there exists a contour around $v$ with respect to $\omega$ which 
does not have good enough excess energy in the sense that 
\begin{equation}
\begin{split}
B_v&:=\bigcup_{\gamma:\gamma \ni v}B(\gamma) \text{ where }
B(\gamma):=\{\omega: | D(\omega)\cap E(\gamma) | 
\geq \delta_0 |\gamma|\}.
\end{split}
\end{equation}

\end{defn}

\begin{defn} Suppose that $\mu$ is a homogeneous tree-indexed Markov chain for a Hamiltonian in the class \eqref{potential-central} fulfilling the bounds \eqref{uU-condition} and \eqref{hyp-psi}. Let $p_1=p_1(\mu)$ be its lazyness parameter as  in Defintion \ref{def-p1}. Let $\delta_0$ as in \eqref{def-delta0}. Define $\lambda(p_1)=\lambda(p_1,u,U,d)$ as
	\begin{equation}
		\begin{split}
			\label{def-lambda-p1}
			&e^{-\lambda(p_1)}:= 
			\inf_{t\geq 0}e^{-t \delta_0}
			(p_1 e^t+1-p_1)^{(d+1)}
			=\left(\frac{p_1}{\delta_0}\right)^{\delta_0}\Bigl(\frac{1-p_1}{d+1-\delta_0} \Bigr)^{1-\delta_0}.
		\end{split}
	\end{equation}
	Note that $\lambda(p_1)\uparrow\infty$ or equivalently $e^{-\lambda(p_1)}\downarrow 0  \text{ as } 
	p_1\downarrow 0$. In particular, if $\mu_\beta$ is the free state of the $q$-clock model at inverse temperature $\beta$ (see \eqref{p1-free-clock}),
	$$\lambda(p_1(\mu_\beta))\uparrow \infty  \text{ as } 
	\beta\uparrow \infty.$$
\end{defn}

\subsection{Key lemma 
}\label{sec-proofs-1}

\begin{lemma}[Key Lemma] \label{key-lemma}

Let $\mu$ be the central state of any model {with Hamiltonian} in the class \eqref{potential-central} fulfilling the bounds \eqref{uU-condition} and \eqref{hyp-psi}. Then for $\beta$ large enough we have, for $\mu$-almost every\ $\omega$
{
\begin{equation}
\begin{split}
\label{eq-bad-dec}\pi(\sigma_v\neq \omega_v|\omega)
\leq   1_{B_v}(\omega) + \epsilon_1(\beta)
\end{split}
\end{equation}
where the event $B_v$ is defined in \eqref{def-bad}, 
and 
$\epsilon_1(\beta)\downarrow0$ as $\beta\uparrow\infty$.
%
Moreover, 
\begin{equation}\label{proba-bad-contour}
	\mu(B(\gamma))\leq e^{-\lambda(p_1)|\gamma|}
\end{equation}
where $\lambda(p_1)$ is defined in \eqref{def-lambda-p1} and for $p_1$ small enough, for any $v\in V$
\begin{equation}\label{proba-bad-at-v}
	m(p_1):=\mu(B_v)\leq \epsilon_2(p_1)
\end{equation}
where $\epsilon_2(p_1)\downarrow0$ as $p_1\downarrow0$.
}
\end{lemma}

\proof {We decompose 
\begin{align}
	\pi(\sigma_v\neq \omega_v|\omega)=\pi(\sigma_v\neq \omega_v|\omega)(1_{B_v}(\omega)+1_{B^c_v}(\omega))
\end{align}
and} we use the {infinite-volume} version of the Peierls bound of Lemma \ref{lem-Peierls} (removing the restriction to finite volume always gives an upper bound) applied to the complement of a bad event at $v$ to bound 
the probability of a mismatch at $v$ by the sum over contours attached at $v$. 
This gives  
\begin{equation}
\begin{split}
\label{XSpSOS2-4l}&
\pi(\sigma_v\neq \omega_v|\omega)1_{B^c_v}(\omega)
\leq \sum_{\bar \gamma, \gamma \ni v}\rho[\omega](\bar \gamma)1_{B^c_v}(\omega)
\end{split}
\end{equation}
where $\rho[\omega](\bar \gamma)$ is the contour activity relative to $\omega$. 
Note that $\omega$ is not assumed to be a local ground state ({in the sense of \cite{CoKuLe22}}), but we may 
still use the bound on the excess energy of a local perturbation {given in Lemma \ref{lem-EEL}}. \\
This delivers, using the assumptions on the potential in the class \eqref{potential-central} and the Definition \ref{def-bad} of $B_v$,
\begin{equation}
\begin{split}
&\rho[\omega](\bar \gamma)1_{B^c_v}(\omega)\cr
&\leq \exp({-\beta (d-1)u |\gamma| + \beta(u+U) | D(\omega)\cap E(\gamma) | +{\Vert\Psi\Vert}_\infty|\gamma|
})1_{B^c_v}(\omega)\cr
&\leq \exp({-\beta |\gamma|((d-1)u - (u+U)\delta_0 -{\Vert\Psi\Vert}_\infty)})\cr
&=\exp({-\beta |\gamma| {(d-1)u}/4 })
\end{split}
\end{equation} 
Using the {entropy bound of Lemma \ref{lem-entropy}} 
we deduce that there exists $C',c'>0$ such that
\begin{equation}
\begin{split}
\label{def-epsilon1-beta}&
\pi(\sigma_v\neq \omega_v|\omega)1_{B^c_v}(\omega)
\leq \sum_{\bar \gamma, \gamma \ni v}e^{-\beta c |\gamma|}
{\leq\sum_{\ell\geq1}(d+1)^{2(\ell-1)}(q-1)^\ell e^{-\beta  \ell (d-1)u/4}}\cr
&\leq C' e^{-c' \beta}=:\epsilon_1(\beta)
\end{split}
\end{equation}
for $\beta>4\log((q-1)(d+1)^2)/(u(d-1))$, which proves \eqref{eq-bad-dec}.

Next we use the exponential Markov inequality to bound $\mu(B_v(\gamma))$ 
and sum over the contours afterwards to bound 
$\mu(B_v)$, as will be described now. 
\begin{align}
	\mu(B_v(\gamma))
	&=\mu\bigl(	 | D(\omega)\cap E(\gamma) | 
	\geq \delta_0 |\gamma|\bigr)\\
	&\leq \inf_{t\geq 0}e^{-t \delta_0 |\gamma|}\mu(e^{t  | D(\omega)\cap E(\gamma) | 
		}) 
\end{align}

It suffices to fix any $x\in \gamma$, look at the conditional measure  
$\mu_a(\cdot ):=\mu(\cdot |\sigma_x=a )$ for any fixed $a\in \Z_q$, and find an upper bound which does not depend on $a$.

Using the Markov chain property of the measure $\mu_a$, successive applications of the homogeneous transition matrix 
$P$ yield
\begin{equation}
\begin{split}
\label{exp-Markov-Bv}
&\mu_a(e^{t  | D(\omega)\cap E(\gamma) | })\leq (p_1 e^t+1-p_1)^{(d+1)|\gamma|}  
\end{split}
\end{equation}
which follows from $$\max_{a\in\Z_q} \sum_{b\neq a}P(a,b)e^t +P(a,a) \leq p_1 e^t + 1-p_1$$ 
and the bound $|E( \gamma) |\leq (d+1)|\gamma|.$
Optimizing over $t$ leads to the bound 
\begin{equation}
\begin{split}
\label{XSpSOS2-4i}
&\mu(B_v(\gamma)){\leq \Bigl(\inf_{t\geq 0}e^{-t \delta_0}
(p_1 e^t+1-p_1)^{(d+1)}  \Bigr)^{|\gamma|}}= e^{-\lambda(p_1)|\gamma|}
\end{split}
\end{equation}
which proves \eqref{proba-bad-contour}, since the definition \eqref{def-lambda-p1} of $\exp(-\lambda(p_1))$ is made to be the value of the above infimum.
Finally using the {entropy bound of Lemma \ref{lem-entropy}} on the number of contours of fixed length 
which are attached to a given point, the estimate for 
$\mu(B_v)$ follows. Indeed, there exists $C>0$ such that for any $c\in(0,1)$ we have
{
\begin{align}\label{def-epsilon2-p1}
	m(p_1):=\mu(B_v)&\leq
	\sum_{\gamma:\gamma\ni v}\mu(B_v(\gamma))
	\leq \sum_{\gamma:|\gamma|\geq 1}e^{-\lambda(p_1)|\gamma|}\nonumber\\
	&\leq \sum_{\ell\geq1}(d+1)^{2(\ell-1)} (q-1)^\ell e^{-\lambda(p_1)\ell}\cr
		&\leq C e^{-c\lambda(p_1)}=:\epsilon_2(p_1)
\end{align}
for $p_1$ sufficiently small, which proves \eqref{proba-bad-at-v}.}
\endproof

%
%

\subsection{Positivity of the Edwards-Anderson parameter}\label{proof-thm-EA}
\begin{proof}[Proof of Theorem \ref{thm-EA}]
	 
The proof is short and elementary, given the reconstruction bound of Theorem \ref{thm-reconstruction}. 
Observe that 
\begin{equation}\label{qEA1}
	\begin{split}
	&q_{\rm EA}^\mu=\sum_{a \in \Z_q}\Bigl(\mu\Bigl( \pi( {\sigma_0=a} | \cdot)^2
\Bigr)	- 
	\mu(\sigma_0=a)^2 \Bigr)\cr
	&=\sum_{a\in \Z_q}\Bigl(\sum_{b \in \Z_q}\mu(\sigma_0=b)\mu\Bigl( \pi( {\sigma_0=a} | \cdot)^2
|\sigma_0=b\Bigr)	- 
	\mu(\sigma_0=a)^2 \Bigr)\cr
	&\geq \sum_{a \in \Z_q}\Bigl(\mu(\sigma_0=a)\mu\Bigl( \pi( {\sigma_0=a} | \cdot)^2
|\sigma_0=a\Bigr)	- 
	\mu(\sigma_0=a)^2 \Bigr)\cr
\end{split}
	\end{equation}
which, by the quantitative reconstruction Theorem \ref{thm-reconstruction} gives the lower bound 
	\begin{equation}\label{qEA1}
	\begin{split}
		q_{\rm EA}^\mu&\geq \sum_{a \in \Z_q}\Bigl(\mu(\sigma_0=a)
	(1-\epsilon_1(\beta))^2 (1- \epsilon_2(p_1))
	-\mu(\sigma_0=a)^2 \Bigr)\\
	&\geq(1-\epsilon_1(\beta))^2 (1- \epsilon_2(p_1))
	-\max_{a\in\Z_q}\mu(\sigma_0=a)
\end{split}
	\end{equation}
which proves the theorem. 
\end{proof}

\subsection{Overlap control for typical extremals}\label{proof-BO}

\subsubsection{Exponential decorrelation of bad sites} 

Suppose that $\mu$ is a homogeneous 
tree-indexed {Markov chain}. 
We do not need to assume any invariances of the measure $\mu$ 
under joint transformations of the local spin-spaces (like a discrete clock-rotation 
or even permutation symmetry). 
We will {only} assume that the jump probability $p_1$ (see Definition \ref{def-p1}), which gives a bound on 
the density of broken bonds in the configuration drawn from $\mu$, 
is small enough. {Using this assumption}, we derive decorrelation between two bad events occuring at the sites $u$ and $v$ (see Definition \ref{def-bad}), with an upper bound 
that is exponential small in the distance between $u$ and $v$. 

\begin{lemma}\label{lem-decorr}  {Consider the central state $\mu$ of any model {with Hamiltonian} in the class} \eqref{potential-central} fulfilling the bounds \eqref{uU-condition} and \eqref{hyp-psi}.
Suppose that the transition matrix $P$ of the state $\mu$ is irreducible 
and aperiodic. Recall the Definition \ref{def-bad} of a bad event $B_v$ at $v\in V$ and the notation for its expectation $\mu(B_v)=:m(p_1)$ (which does not depend on $v$).
Let 
\begin{equation}\label{def-cov-uv}
	\Cov(u,v):=\mu\Bigl((1_{B_v}(\omega)-m(p_1))(1_{B_u}(\omega)-m(p_1))\Bigr).
\end{equation}
Then for small enough $p_1$, there exist $C>0$ such that for any $c_1\in(0,\frac13)$ and $c_2\in(0,1)$, for any $u,v\in V$,
\begin{equation}
\begin{split}
\label{exp-decay-cov}
|\Cov(u,v)|
&\leq C e^{- c |v-u|}
\end{split}
\end{equation}
where $|v-u|$ is the graph distance in the tree and
\begin{equation}
\begin{split}
\label{X22}
&e^{- c}=\max\{e^{-c_1\lambda(p_1)},|\lambda_2(P)|^{c_2} \}
\end{split}
\end{equation}
where $\lambda(p_1)$ is defined in \eqref{def-lambda-p1} and $\lambda_2(P)$ is the second largest eigenvalue (in modulus)
of the transition matrix $P$ of $\mu$. 
\end{lemma}

The proof uses a combination of two facts. 
First, long bad contours attached to 
a given site are exponentially improbable,  see Lemma \ref{key-lemma}.  
This explains the occurence of $\lambda(p_1)$.
Second there is exponential relaxation of the one-dimensional 
Markov chain obtained by restriction of $\mu$ to a branch, conditionally on a root, which explains the appearance of the second largest eigenvalue $\lambda_2(P)$ of the transition matrix. Of course $|\lambda_2(P)|<1$.

\proof 

We first define several objects and events which are depicted on Figure~\ref{fig-badbad}.
\begin{figure}
	\begin{center}
	\includegraphics[width=10cm]{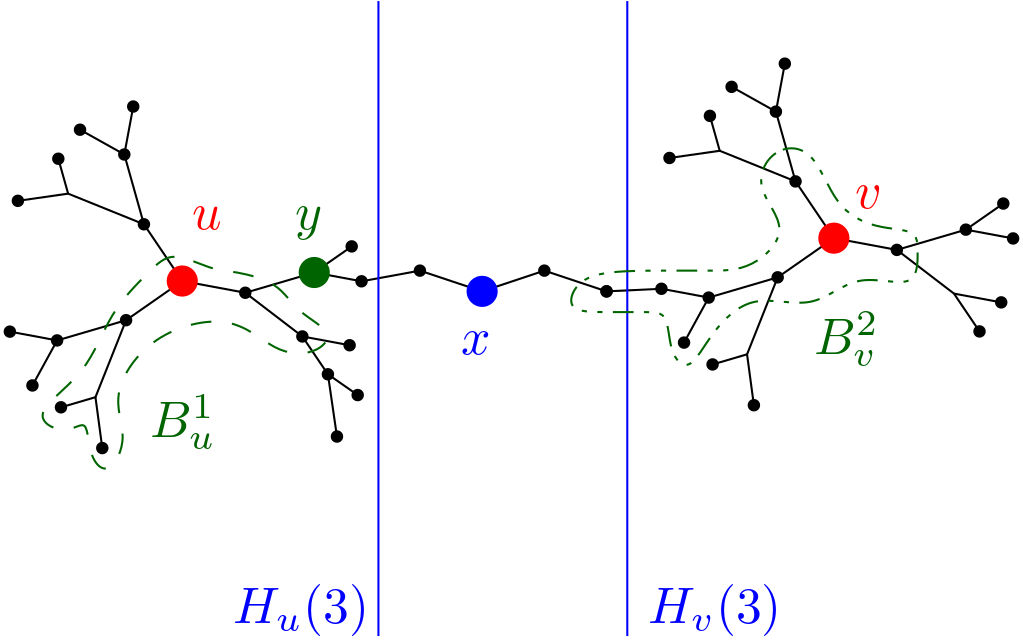}
	\caption{Objects appearing in the proof of Lemma \ref{lem-decorr} of exponential decorrelation of bad events at $u$ and $v$. On the picture $d=2,|u-v|=11, r=3$. Half-trees $H_u(3)$ and $H_v(3)$ are delimited by the two vertical lines. One of the midpoints between $u$ and $v$ is denoted by $x$. The support of the bad event $B_u^1$, depicted with a dashed line, stays inside $H_u(3)$, while the support of the bad event $B_v^2$, depicted with a dotted line, exits $H_v(3)$. The vertex $y$ is the closest to $u$ on the path from $u$ to $x$ which does not touch the support of $B_u^1$. }\label{fig-badbad}
	\end{center}
  \end{figure}
Denote by the half-trees $H_u(r)$ and $H_v(r)$ the sets 
\begin{equation}
\begin{split}
\label{XSpSOS2-66}
&H_u:=H_u(r):=\{x\in V, |x-u|\leq |x-v|-r\}\cr
&H_v:=H_v(r):=\{x\in V, |x-v|\leq |x-u|-r\}\cr
\end{split}
\end{equation}
We will choose $r=\lfloor |u-v|/3\rfloor$. 
Clearly $H_u$ and $H_v$ are disjoint, with their distance growing proportionally 
to $|u-v|$.  
Define 
\begin{equation}
\begin{split}
\label{bad-decomposition}
&B^1_u:=\bigcup_{\gamma: H_u\supset
\gamma \ni u}B_u(\gamma),\quad B^2_u:=B_u\backslash B^1_u\cr
&B^1_v:=\bigcup_{\gamma: H_v\supset
\gamma \ni u}B_v(\gamma),\quad B^2_v:=B_v\backslash B^1_v.\cr
\end{split}
\end{equation}
The sets $B^1_u$ and $B^1_v$ describe bad events at $u$ (resp.\ $v$) which are produced by contours
included in $H_u$ (resp.\ $H_v$). 
The sets $B^2_u$ and $B^2_v$ describe bad events which are not produced by the previous contours and may 
connect $u$ and $v$. \\

Bad event caused by long contours are exponentially improbable 
in the minimal length of such contours. Namely, by the Key Lemma \ref{key-lemma} and the entropy bound given in Lemma \ref{lem-entropy}, there exist $C_1>0$ such that for any $c_1\in(0,\frac13)$ we have
{
\begin{equation}
\begin{split}
\label{}
\mu(B^2_u)=\mu(B^2_v)&\leq 
\sum_{\gamma: \gamma \ni u, \gamma \cap H^c_u\neq \emptyset}\mu(B_u(\bar\gamma))
\leq \sum_{\gamma: \gamma \ni u, |\gamma|\geq|u-v|/3}
e^{-\lambda(p_1)|\gamma|}\\
&\leq \sum_{\ell\geq|u-v|/3}(d+1)^{2(\ell-1)}(q-1)^\ell e^{-\lambda(p_1)\ell}\\
&\leq C_1 \exp({-c_1\lambda(p_1)|u-v|}).
\end{split}
\end{equation}}
Spelling out the correlation function between the bad events in terms 
of the decomposition \eqref{bad-decomposition}, we therefore have 
\begin{equation}
\begin{split}
\label{XSpSOS2-99}
&|\Cov(u,v)|=|\mu(B^1_u;B^1_v)+\mu(B^1_u;B^2_v)+\mu(B^2_u;B^1_v)+\mu(B^2_u;B^2_v)|\cr
&\leq |\mu(B^1_u;B^1_v)|+ 3 \min\{\mu(B^2_u),\mu(B^2_v)\}\\
&\leq |\mu(B^1_u;B^1_v)|+ 3C_1 \exp({-c_1\lambda(p_1)|u-v|}).\cr
\end{split}
\end{equation}
It remains to control the decay of $|\mu(B^1_u;B^1_v)|$. 
We would have the exact decorrelation 
$\mu(B^1_u;B^1_v)=0$ in the free state of a clock model, but 
for an non-symmetric model there is no reason to assume that such 
decorrelation holds.  
Nevertheless, we may use the {Markov chain} property of $\mu$ along the path from $u$ to $v$ and 
the exponential convergence of the corre{s}ponding {1d-}Markov chain, to arrive 
at the decorrelation which is exponential in the distance $|u-v|$ (with our irreducibility and aperiodicity assumptions).

Indeed, choose $x=x(u,v)$ to be the (or one of the at most two) middle sites
between $u$ and $v$ (which means that the distances $|u-x|$ and $|v-x|$ differ at most by one).
From the tree-indexed Markov chain property of $\mu$, using conditional independence, we have 
\begin{equation}
\begin{split}
\label{XSpSOS2-yy}
&\mu(B^1_u\cap B^1_v)=\sum_{a\in\Z_q} \mu(B^1_u|\sigma_x=a)\mu(B^1_v|\sigma_x=a)\mu(\sigma_x=a).
\end{split}
\end{equation}
Let $\Delta_{u,x}(a):=\mu(B^1_u|\sigma_x=a)-\mu(B^1_u)$ and similarly for $v$. We have 
 \begin{align}
\label{XSpSOS2-jj}
|\mu(B^1_u; B^1_v)|&	=
\big|\sum_{a\in\Z_q} \Bigl(\mu(B^1_u)\Delta_{v,x}(a)+\mu(B^1_v)\Delta_{u,x}(a)+\Delta_{u,x}(a)\Delta_{v,x}(a)\Bigr)\mu(\sigma_x=a)\big|\nonumber\\
&\leq2\max_{a\in\Z_q}|\Delta_{u,x}(a)|+\max_{a\in\Z_q}|\Delta_{v,x}(a)|.
\end{align}
By symmetry it suffices to bound $|\Delta_{u,x}(a)|$. 
Choose $y$ on the path $[x,u]$, between $x$ and $u$, such that it is the closest to $u$ with the property that all the contours appearing in the definition of $B^1_u$ do not touch 
$(x,y]$. 
To control the difference between $\mu(B^1_u|\sigma_x=a)$ and $\mu(B^1_u)$
we use the exponential relaxation of the homogenous 1d Markov chain which arises as restriction of $\mu$, on the path from  $x$ to $y$.
There exists $C_2>0$ such that for any $c_2\in(0,1)$, 
\begin{equation}
\begin{split}
\label{exp-decorrelation-1dMC}
&|\Delta_{u,x}(a)|=|\sum_{b\in\Z_q} \mu(B^1_u|\sigma_y=b)( \mu(\sigma_y=b|\sigma_x=a) -\mu(\sigma_y=b))|\cr
&\leq \sum_{b\in\Z_q} |\mu(\sigma_y=b|\sigma_x=a) -\mu(\sigma_y=b))|\cr
&=\sum_{b\in\Z_q} |P^{|x-y|}(a,b) -\mu(\sigma_y=b))|\cr
&\leq C_2 |\lambda_2(P)|^{c_2|x-y|}
\end{split}
\end{equation}
where the last line follows e.g.\ from \cite[Example 4.3.9: 
Rates of Convergence via the Perron-Frobenius Theorem]{Bremaud}. 
As by construction $|y-x|\sim \frac{1}{6}|u-v|$, as $|u-v|\uparrow\infty$, this implies the 
desired exponential decorrelation of  $|\mu(B^1_u;B^1_v)|$, and concludes the proof. 
\endproof

\subsubsection{Almost sure convergence of empirical means  of bad sites  
}

As a consequence of the exponential decorrelation of bad events we 
harvest the following lemma. 
\begin{lemma} \label{lem-meanbad}
	Consider the central state $\mu$ of any model {with Hamiltonian of the form}  \eqref{potential-central} fulfilling the bounds \eqref{uU-condition} and \eqref{hyp-psi}.
	{For any {sequence of subsets} $\Lambda_n=\Lambda_n^r$ of a branch of the tree such that $\sum_{n=1}^\infty 
	{|\Lambda_n|}^{-1}<\infty$ and}
for $\mu$-almost every $\omega$, the {following} limit holds 
\begin{equation}
\begin{split}
\label{XSpSOS2-hju}
&\lim_{n\uparrow \infty}
\frac{1}{|\Lambda_n|}\sum_{v\in \Lambda_n}1_{B_v}(\omega)=m(p_1)
\end{split}
\end{equation}
where $m(p_1)=\mu(B_0)$.
\end{lemma} 
While the limit in question, if it exists is necessarily tail-trivial, the statement is non-trivial 
as the measure $\mu$ is not extremal at low temperature.  
So it must be proved manually.

\proof 
The almost sure limit holds, if we can 
show that for any fixed $\delta>0$ the number of indices
$n$ for which 
\begin{equation}
\begin{split}
\label{XSpSOS2-yuy}
&\frac{1}{|\Lambda_n|}\Bigl | \sum_{v\in \Lambda_n}
(1_{B_v}(\omega)-m(p_1))\Bigr | \geq \delta 
\end{split}
\end{equation}
is finite, for $\mu$-almost every $\omega$. 
By Borel-Cantelli it suffices to show that 
\begin{equation}
\begin{split}
\label{bb}
&\sum_{n=1}^\infty
\mu\Bigl(\frac{1}{|\Lambda_n|}\Bigl |\sum_{v\in \Lambda_n}
(1_{B_v}(\omega)-m(p_1))\Bigr | > \delta \Bigr)<\infty.	
\end{split}
\end{equation}
For this we use the quadratic Chebychev inequality to bound the probability 
in the above sum by 
\begin{equation}
\begin{split}
\label{jki}
& \frac{1}{|\Lambda_n|^2 \delta^2}\sum_{v,u\in \Lambda_n}
\mu\Bigl((1_{B_v}(\omega)-m(p_1))(1_{B_u}(\omega)-m(p_1))\Bigr).
\end{split}
\end{equation}
Recalling the definition \eqref{def-cov-uv} of site covariances $\Cov(u,v)$ and using Lemma \ref{lem-decorr}, we deduce that there exists $C>0$ such that 
\begin{equation}
\begin{split}
\label{mll}
&\sum_{n=1}^\infty\frac{1}{|\Lambda_n|^2}\sum_{v,u\in \Lambda_n}\Cov(u,v)\cr
&\leq \sum_{n=1}^\infty
\Biggl(\frac{1}{|\Lambda_n|}\Cov(0,0)+ \frac{2}{|\Lambda_n|^2} 
\sum_{v\in \Lambda_n}\sum_{u\in \Lambda_n, u>v}\Cov(u,v)\Biggr)\cr
&\leq C \sum_{n=1}^\infty\frac{1}{|\Lambda_n|}<\infty
\end{split}
\end{equation}
since for any fixed vertex $v\in V$, the sum $\sum_{u>v}\Cov(u,v)\leq C'<\infty$ uniformly in $v$ by Lemma \ref{lem-decorr}.
This proves the existence of the almost sure limit of empirical 
sums of indicators of bad points, along the quadratic volume sequences defined in \eqref{lambda-n-r}. 
\endproof

\subsubsection{Concentration of spin-overlaps  under typical extremals}\label{GoodBad}

\begin{lemma}\label{lem-overlap} 
	Consider the central state $\mu$ of any model {with Hamiltonian} in the class \eqref{potential-central} fulfilling the bounds \eqref{uU-condition} and \eqref{hyp-psi}.\\
	There exist a sequence of spacings $r=(r_i)_{i\in \N}$, 
	in general depending on model parameters, such that
	for any {sequence of subsets} $\Lambda_n=\Lambda^r_n$ of a branch of the tree such that $\sum_{n=1}^\infty 
	{|\Lambda_n|}^{-1}<\infty$, the following holds. \\For $\mu$-almost every $\omega$, for $\pi(\cdot |\omega)$-almost every realizations of $\sigma$ 
we have  
\begin{equation}
\begin{split}
\label{limit-concentration}
\limsup_{n\uparrow\infty}\frac{1}{|\Lambda_{n}|}\Bigl| 
\sum_{v\in \Lambda_{n}}(1_{\sigma_v\neq\omega_v}- \pi(\sigma_v\neq\omega_v|\omega))\Bigr|=0.
\end{split}
\end{equation}
\end{lemma}


\proof 
It suffices to prove that, for any $\delta>0$ we have, for $\mu$-a.e. $\omega$,
\begin{equation}
\begin{split}
\label{zsz}
\pi\left(\frac{1}{|\Lambda_{n}|}\Bigl|\sum_{v\in \Lambda_{n}}
(1_{\sigma_v\neq\omega_v}-\pi(\sigma_v\neq\omega_v |\omega))\Bigr|>\delta \text{ for infinitely many }n \,\middle|\,\omega\right)=0
\end{split}
\end{equation}  
Using the Borel-Cantelli lemma it suffices to show that, for $\mu$-a.e. $\omega$,
\begin{equation}
\begin{split}
\label{zaqs}
\sum_{n=1}^\infty\pi\left(\frac{1}{|\Lambda_{n}|}\Bigl|\sum_{v\in \Lambda_{n}}
(1_{\sigma_v\neq\omega_v}-\pi(\sigma_v\neq\omega_v |\omega))\Bigr|>\delta
\,\middle|\,\omega\right)<\infty.
\end{split}
\end{equation}
Now by the quadratic Chebychev inequality, it is enough to 
show that
\begin{equation}
\begin{split}
\label{sum-finite}
\sum_{n=1}^\infty \frac{1}{|\Lambda_{n}|^2}
\sum_{v,u\in \Lambda_{n}}|\pi(\sigma_v\neq\omega_v;\sigma_u\neq\omega_u
|\omega)|<\infty
\end{split}
\end{equation}
{holds} for $\mu$-a.e. $\omega$.  
We would like to use the abstract fact that 
$\pi(\cdot|\omega)$ is an extremal Gibbs measure, and hence decorrelates {in the sense of (\ref{short})}. Indeed, by extremality of $\pi(\cdot |\omega)$ at fixed $\omega$, for any fixed $u$ and any sequence $v\uparrow\infty$ 
(meaning that $v$ leaves any finite ball) the following non-quantative result holds: For $\mu$-a.e.\ $\omega$,
\begin{equation}
\begin{split}
\label{kiuj}
\lim_{v\uparrow\infty}|\pi(\sigma_v\neq\omega_v;\sigma_u\neq\omega_u|\omega)|=0.
\end{split}
\end{equation}
The problem is that the rate of convergence may depend on $\omega$, {while 
we want a speed which is uniform in all the extremals (in order for the limit \eqref{limit-concentration} to hold along the same sparse sequence of volumes for $\mu$-a.e. $\omega$).} {An explicit analysis 
is difficult, due to the lack of spatial homogeneity.}
To bypass this difficulty, note that to ensure \eqref{sum-finite} it is sufficient to ensure that 
\begin{equation}
\begin{split}
\label{zedf}
\int\mu(d\omega)\sum_{n=1}^\infty \frac{1}{|\Lambda_{n}|^2}
\sum_{v,u\in \Lambda_{n}}|\pi(\sigma_v\neq\omega_v;\sigma_u\neq\omega_u
|\omega)|<\infty.
\end{split}
\end{equation}
Using monotone convergence for non-negative functions 
to interchange the $\mu$-integral and the infinite sum, 
we can deduce this from 
\begin{equation}
\begin{split}
\label{erty}
\sum_{n=1}^\infty \frac{1}{|\Lambda_{n}|^2}
\sum_{v,u\in \Lambda_{n}}\underbrace{\int\mu(d\omega)|\pi(\sigma_v\neq\omega_v;\sigma_u\neq\omega_u
|\omega)|}_{=:{c(u,v)}}<\infty
\end{split}
\end{equation}
By dominated convergence we have, for any $u\in V$,
\begin{equation}
\begin{split}
\label{lopm}
{\lim_{v\uparrow \infty}c(u,v)=0}
\end{split}
\end{equation}
{
We look at the off-diagonal terms in the double-sum \eqref{erty} as in \eqref{mll}. 
Now we can achieve the desired summability of 
\begin{equation}
\begin{split}
\label{offdia}
\sum_{n=1}^\infty \frac{1}{|\Lambda_n|^2}
\sum_{i=1}^{|\Lambda_n|}\sum_{j=1}^{i-1}c(v_j,v_i)<\infty
\end{split}
\end{equation}
by iteratively choosing $v_i$ given $v_1, \dots, v_{i-1}$ on the same branch so large 
that 
$$\sum_{j=1}^{i-1}c(v_j,v_i)<i^{-2}.$$ 
The above choice provides the following upper bound of \eqref{offdia} as 
$$\sum_{n=1}^\infty \frac{1}{|\Lambda_n|}\sum_{i=1}^{\infty}i^{-2}$$
which is finite. \endproof }

\subsubsection{Proof of Theorem \ref{thm-branch-overlap}}\label{proof-thm-branch-overlap}
\proof
Applying Lemma \ref{lem-overlap}, the Key Lemma \ref{key-lemma}, and Lemma \ref{lem-meanbad}, we deduce that for sparse enough sets $\Lambda_n=\Lambda_n^{\rm r}$ the following holds.
For $\mu$-almost every $\omega$, $\pi(\cdot|\omega)${-}almost surely,
\begin{align*}
	\liminf_{n\uparrow\infty}\frac{1}{|\Lambda_n|}\sum_{v\in \Lambda_n}1_{\sigma_v=\omega_v}
	&\geq 1-\limsup_{n\uparrow\infty}\frac{1}{|\Lambda_n|}\sum_{v\in \Lambda_n}\pi(\sigma_v\neq\omega_v|\omega)\\
	&\geq 1-\liminf_{n\uparrow\infty}\frac{1}{|\Lambda_n|}\sum_{v\in \Lambda_n}1_{B_v}(\omega)-\epsilon_1(\beta)\\
	&= 1-\epsilon_1(\beta)-m(p_1)\\
	&\geq 1-\epsilon_1(\beta)-\epsilon_2(p_1)
\end{align*}
with $\epsilon_1(\beta)=C'e^{-c'\beta}$, see \eqref{def-epsilon1-beta}, and $\epsilon_2(p_1)=Ce^{-c\lambda(p_1)}$, see \eqref{def-epsilon2-p1} and Definition~\ref{def-p1}. This proves the theorem.
\endproof

\subsection{From overlap control to almost sure singularity of extremals}\label{proof-cor-glassy}

\proof[Proof of Corollary \ref{cor-glassy}] 
On the one hand we have the following lower bound from Theorem \ref{thm-branch-overlap} and monotone convergence. For sparse enough sets $\Lambda_n=\Lambda_n^{\rm r}$, for $\mu$ almost every $\omega$,
\begin{equation}
\begin{split}
\label{}
&\pi(\underline\phi^\omega|\omega)
=\liminf_{n\uparrow\infty}\frac{1}{|\Lambda_n|}\sum_{v\in \Lambda_n}\pi(\sigma_v=\omega_v|\omega)
\geq 1-\epsilon_1(\beta)-\epsilon_2(p_1).
\end{split}
\end{equation}
On the other hand, we use
\begin{equation}
\begin{split}
\label{sedf}
1_{\sigma_v=\omega_v}\leq 1_{\omega_v=\omega'_v} +1_{\sigma_v\neq\omega'_v}
\end{split}
\end{equation}
to write the upper bound
\begin{equation}
\begin{split}
\label{dcvbg}
\pi(
\underline\phi^\omega|\omega')\leq \limsup_{n\uparrow\infty}\frac{1}{|\Lambda_n|}\sum_{v\in \Lambda_n}1_{\omega_v=\omega'_v} +
\pi\left(\limsup_{n\uparrow\infty}\frac{1}{|\Lambda_n|}\sum_{v\in \Lambda_n}1_{\sigma_v\neq\omega'_v}\middle|\omega'\right).
\end{split}
\end{equation}

The first term in the r.h.s.\ is equal to ${1}/{q}$ for $\mu\otimes\mu$-a.e.\ pair $(\omega,\omega')$, in the case of clock models. Indeed, the Borel-Cantelli argument, together with the Chebychev inequality as in the proof of Lemma \ref{lem-meanbad} 
shows that the desired strong law of large numbers follows (since exponential decorrelation of the local events $(1_{\omega_v=a})_{v\in\Lambda_n^r}$ follows from exponential decorrelation of the 1d Markov chains as in \eqref{exp-decorrelation-1dMC}).}
In the general case of central 
states this term equals $\sum_{a\in \Z_q}\mu(\sigma_0=a)^2$; it is close to $1/q$ 
for perturbed free states of clock models, see Section \ref{sec-small-p1}.

The second term in the r.h.s.\ is $\pi(\cdot | \omega')$-a.s. controlled 
via Theorem \ref{thm-branch-overlap} by 
\begin{equation}
\begin{split}
\label{dfg}
\limsup_{n\uparrow\infty}\frac{1}{|\Lambda_n|}\sum_{v\in \Lambda_n}1_{\sigma_v\neq\omega'_v}\leq \epsilon_1(\beta)+\epsilon_2(p_1).
\end{split}
\end{equation}
So the statement of the theorem follows once 
$$1-\epsilon_1(\beta)-\epsilon_2(p_1)>\sum_{a\in \Z_q}\mu(\sigma_0=a)^2+\epsilon_1(\beta)+\epsilon_2(p_1)$$ which concludes the proof. 
\endproof
\vspace{2cm}

\paragraph*{Acknowledgements.}
C.K. and A.L.N. thank Labex B\'ezout (ANR-10-LABX-58)
and Laboratory LAMA (UMR CNRS 8050) at Universit\'e Paris Est Cr\'eteil (UPEC) for 
various supports. Research of A.L.N. and L.C. have also been supported by the CNRS IRP (International Research Project) EURANDOM ``{\it Random Graph, Statistical Mechanics and Networks}'' and by the LabEx PERSY-VAL-Lab (ANR-11-LABX-0025-01) funded by the French program Investissement d’avenir.

\end{document}